\documentclass[10pt,DIV=11,a4paper]{scrartcl}

\usepackage[pdfusetitle,hidelinks]{hyperref}

\usepackage{amsmath}
\usepackage{amsthm}
\usepackage{amssymb}
\usepackage{mathtools}
\usepackage{newtxmath}

\usepackage{color}

\usepackage{enumitem}
\setlist{noitemsep}

\usepackage[capitalise]{cleveref}
\usepackage[backend=bibtex,isbn=false,maxcitenames=5,maxnames=99]{biblatex}
\DeclareSourcemap{
  \maps[datatype=bibtex]{
    \map{
      \step[fieldsource=doi,final]
      \step[fieldset=url,null]
      \step[fieldset=urldate,null]
    }
  }
}

\bibliography{lit.bib}

\usepackage{tikz}
\usetikzlibrary{patterns}
\usepackage{pgfplots}
\pgfplotsset{compat=1.14}
\pgfplotsset{
    tick align=outside,
    x grid style={white},
    xmajorgrids,
    y grid style={white},
    ymajorgrids,
    axis line style={white},
    axis background/.style={fill=white!92!black},
    legend style={draw=white, fill=white},
     legend cell align={left}
}

\newtheorem{thm}{Theorem}
\newtheorem{proposition}[thm]{Proposition}
\newtheorem{lemma}[thm]{Lemma}

\theoremstyle{remark}
\newtheorem{remark}[thm]{Remark}

\newcommand{\tnorm}[1]{{\left\vert\kern-0.25ex\left\vert\kern-0.25ex\left\vert #1
    \right\vert\kern-0.25ex\right\vert\kern-0.25ex\right\vert}}

\newcommand{\R}{\mathbb{R}}
\newcommand{\N}{\mathbb{N}}

\newcommand{\fsL}{\textnormal{L}} 
\newcommand{\fsH}{\textnormal{H}} 
\newcommand{\fsW}{\textnormal{W}} 
\newcommand{\fsC}{\mathscr{C}} 

\newcommand{\ee}{\mathrm{e}} 
\newcommand{\dd}{\mathrm{d}} 

\newcommand{\init}{\mathrm{in}} 

\newcommand{\ind}{\vmathbb{1}} 
\DeclareMathOperator{\supp}{supp}

\DeclareMathOperator{\oscillation}{osc}
\DeclareMathOperator{\diam}{Diam}

\title{Global strong solutions for the triangular Shigesada-Kawasaki-Teramoto cross-diffusion system in three dimensions and parabolic regularisation for  increasing functions}

\author{Hector Bouton
  \thanks{Email: \href{hector.bouton@ens.psl.eu}
    {\texttt{hector.bouton@ens.psl.eu}}\\
  Université Paris Cité and Sorbonne Université, CNRS, IMJ-PRG, F-75013 Paris, France}
  \and
  Laurent Desvillettes
  \thanks{Email: \href{mailto:desvillettes@imj-prg.fr}
    {\texttt{desvillettes@imj-prg.fr}}
    \\
    Université Paris Cité and Sorbonne Université, CNRS and IUF, IMJ-PRG, F-75013 Paris, France.}
   \and
  Helge Dietert
   \thanks{Email: \href{mailto:helge.dietert@imj-prg.fr}
    {\texttt{helge.dietert@imj-prg.fr}} \\
    Université Paris Cité and Sorbonne Université, CNRS, IMJ-PRG, F-75013 Paris, France.}
}
\begin{document}
\maketitle

\begin{abstract}
  We prove the existence of global strong solutions to the triangular
  Shigesada-Kawasaki-Teramoto (SKT) cross-diffusion system with Lokta-Volterra
  reaction terms in three dimensions.  A key part is the independent careful
  study of the parabolic equation \(a\partial_t w - \Delta w = f\) with a rough coefficient
  \(a\), homogeneous Neumann boundary conditions, and the special assumption
  \(\partial_tw \ge 0\).  By the same method, we obtain estimates for solutions to reaction-diffusion
  systems modelling reversible chemistry.
\end{abstract}

\section{Introduction}

The Shigesada-Kawasaki-Teramoto (SKT) system
\cite{shigesada-kawasaki-teramoto-1979-spatial} is a popular cross-diffusion
model from Biology where it is used in population dynamics
\cite{desvillettes-2024-about-shigesada-kawasaki-teramoto}.  A common form is a
triangular cross-diffusion system with two unknowns
$u:=u(t,x) \ge 0, v:=v(t,x) \ge 0$ modelling the concentration of two species.  The
evolution is prescribed by
\begin{equation}
  \label{eq:SKT}
  \left\{
  \begin{aligned}
    \partial_t u - \Delta[ (d_1 +\sigma v) u ] = f_u (u,v), \\
    \partial_t v - d_2\,\Delta v = f_v (u,v), \\
  \end{aligned}
  \right.
\end{equation}
where $d_1,d_2 >0$ are constant diffusion rates, $\sigma >0$ is a coupling constant,
and \(f_u\) and \(f_v\) are Lotka-Volterra type terms modelling competition,
that is
\begin{align} \label{SKTge1}
  f_u(u,v) &= u\,( r_u - d_{11}\, u - d_{12}\, v), \\ \label{SKTge2}
  f_v(u,v) &=  v\,(r_v - d_{21}\, u - d_{22}\, v),
\end{align}
where $r_u,r_v\ge 0$, $d_{ij} > 0$, for $i,j=1,2$.  We also write
\(\mu := d_1+\sigma v\).

The system is defined on a smooth bounded domain $\Omega$ of $\R^d$, and endowed with
homogeneous Neumann boundary conditions
\begin{equation}\label{skt_bc}
  \nabla v \cdot \vec n(x) = 0 \quad \text{ and } \nabla [(d_1+\sigma v) u ] \cdot \vec n(x) = 0  \quad  \quad \text{ for } (t,x) \in \R_+ \times \partial\Omega,
\end{equation}
together with initial  conditions:
\begin{equation}\label{skt_id}
   u(0, \cdot) = u_\init , \qquad \qquad v(0, \cdot) = v_\init  .
\end{equation}
\medskip

The existence of global strong solutions in three dimensions has been a long
open problem even though there is a strong interest in this cross-diffusion
system.  Previously, the theory of global strong solutions was only developed in
one or two dimensions in
\textcite{lou-ni-wu-1998,desvillettes-2024-about-shigesada-kawasaki-teramoto}.
For higher dimensions there were only results (about global strong solutions)
when the cross-diffusion is small, see
\textcite{choi-lui-yamada-2003-existence-shigesada-kawasaki-teramoto}, or when
self-diffusion is added, see
\textcite{choi-lui-yamada-2004-existence-shigesada-kawasaki-teramoto,
  van-2007-global-shigesada-kawasaki-teramoto,van-2008}.  Notice that existence
of weak solutions is easy to obtain in the triangular case for all dimensions
(cf.\ for example
\cite{desvillettes-2024-about-shigesada-kawasaki-teramoto,desvillettes-trescases-2015-new}). The
non-triangular case is completely different and entropy functionals play there a
very important role
(\cite{chen-juengel-2006-analysis,desvillettes-lepoutre-moussa-trescases-2015}).
Finally, we refer to \cite{guerand-menegaki-trescases-2023-global} for an
example of use of De~Giorgi type methods for cross-diffusion equations belonging
to the same class as the SKT system.

This work shows the global existence of strong solutions in dimensions \(d \le 4\) for the triangular system (without self-diffusion and for cross-diffusion of arbitrary size):

\begin{thm}\label{proskt}
  Let $d \le 4$ and $\Omega \subset \R^d$ be a smooth ($\fsC^2$) bounded domain, let
  $d_1,d_2,\sigma >0$, $r_u,r_v\ge 0$, $d_{ij} > 0$ for $i,j=1,2$.  We suppose that
  $u_\init, v_\init \ge 0$ are initial data which lie in $\fsC^2(\overline{\Omega})$ and
  are compatible with the Neumann boundary condition.

  Then, there exists a global (defined on $\R_+ \times \Omega$) strong (that is, all terms
  appearing in the equation are defined a.e.) solution to system \eqref{eq:SKT}--\eqref{skt_id}.
\end{thm}



\medskip

Another consequence of our strategy is a new and very quick proof of existence of strong
global solutions (for general initial data) in dimension $d \le 4$, to a classical
quadratic system of reaction diffusion coming out of reversible chemistry.

Namely, we consider here the unknowns $u_i: = u_i(t,x) \ge 0$, with
$i=1,\dots ,4$, and the set of equations
\begin{equation}\label{u14}
  \partial_t u_i - d_i \,\Delta u_i = (-1)^i \,(u_1\,u_3 - u_2\,u_4),
\end{equation}
with $d_i >0$, defined on a smooth bounded domain $\Omega$ of $\R^d$, together with
homogeneous Neumann boundary conditions and initial conditions:
\begin{equation}\label{u14bc}
  \nabla u_i(t,x) \cdot \vec n(x) = 0\quad \text{ for } (t,x) \in \R_+ \times \partial\Omega, \qquad  u_i(0, \cdot) = u_i^\init ,
\end{equation}
and we prove the following proposition:
\medskip

\begin{thm}\label{pron}
  Let $d \le 4$ and $\Omega \subset \R^d$ be a smooth ($\fsC^2$) bounded domain.  Let $d_i >0$
  be constant diffusion rates, and suppose that the initial data $u_i^\init \ge 0$
  lie in $\fsC^2(\overline{\Omega})$ and are compatible with the homogeneous Neumann
  boundary condition of \eqref{u14bc}.

  Then there exists a global (defined on $\R_+ \times\Omega$) strong (that is, all terms
  appearing in the equation are defined a.e.) solution to system \eqref{u14},
  \eqref{u14bc}.
\end{thm}

Existence of strong solutions for this system in all dimension $d$ was obtained
by different methods in
\cite{caputo-goudon-vasseur-2019-solutions-c,fellner-morgan-tang-2020-global,souplet-2018-global},
following the pioneering works \cite{kanel-1984-cauchys,kanel-1990-solvability}
in the whole space.

The specificity of our new proof is to be extremely short and self-contained.
Moreover, the constants can be explicitly estimated (for example, the quadratic
nonlinearity could be replaced by a nonlinearity behaving like a power $2+\beta$, where
$\beta>0$ can be explicitly estimated in terms of the diffusion rates $d_i$). One
drawback however is that it does not cover the cases when $d\ge 5$. The methods that we use
are much closer to those of
\cite{caputo-goudon-vasseur-2019-solutions-c,fellner-morgan-tang-2020-global}
than those of
\cite{souplet-2018-global,kanel-1984-cauchys,kanel-1990-solvability}.  \medskip

We also provide a more general result (Proposition~\ref{npp}) of existence of
global strong solutions in dimension $d\le 4$ for reaction-diffusion systems with
mass dissipation, improving recent results on the same kind of systems obtained
in \cite{morgan-tang-2020-boundedness-lyapunov,
  fellner-morgan-tang-2020-global,fellner-morgan-tang-2021-uniform}.  More
details and references are provided in \cref{sec:quadratic-intermediate-sums}.
Among the systems which can be treated, we present a system coming from the
reversible chemical reaction $b_1S_1 + ... + b_mS_m \rightleftharpoons S_{m+1} + S_{m+2}$, where
$m \in\N\setminus\{0\}$ and for $i \in 1,\dots,m$, we have
$b_i \in\N\setminus\{0\}$. This system writes
\begin{equation} \label{uum}
  \left\{ \begin{lgathered}
  \partial_t u_{m+1} - d_{m+1} \Delta u_{m+1} = \prod_{i=1}^m u_i^{b_i} - u_{m+1}\,u_{m+2}, \\
  \partial_t u_{m+2} - d_{m+2} \Delta u_{m+2} =  \prod_{i=1}^m u_i^{b_i} - u_{m+1}\,u_{m+2},  \\
    \partial_t u_i - d_i \Delta u_i = -  b_i\,\left( \prod_{i=1}^mu_i^{b_i} - u_{m+1}\,u_{m+2}\right)\quad
    \text{ for }\quad i=1,\dots,m.
 \end{lgathered} \right.
\end{equation}
It is complemented with the homogeneous Neumann boundary condition and initial
data \eqref{u14bc} (for $i=1,\dots,m+2$).  \medskip

Then, we can prove the

\begin{thm}\label{pron2}
  Let $d \le 4$ and $\Omega \subset \R^d$ be a smooth ($\fsC^2$) bounded domain.  Let $d_i >0$
  be constant diffusion rates, and suppose that the initial data $u_i^\init \ge 0$
  lie in $\fsC^2(\overline{\Omega})$ and are compatible with the homogeneous Neumann
  boundary condition of \eqref{u14bc}.

  Then there exists a global (defined on $\R_+ \times\Omega$) strong (that is, all terms
  appearing in the equation are defined a.e.) solution to system \eqref{uum},
  \eqref{u14bc}.
\end{thm}
\bigskip

A key step of our approach is the study of solutions \(u :=u(t,x)\) to the equation
\begin{equation*}
  \partial_t u - \Delta(\mu u) = f(t,x),
\end{equation*}
where \(f\) is some forcing and \(\mu :=\mu(t,x)\) are given diffusion coefficients
for which we know only a lower and upper bound (rough coefficients).  Such an equation can be studied thanks
to 
duality estimates, see
\cite{canizo-desvillettes-fellner-2014-improved,lepoutre-moussa-2017-entropic}.
However, we will rather follow
\cite{fellner-morgan-tang-2020-global,caputo-goudon-vasseur-2019-solutions-c},
and introduce a new unknown \(w(t,x) := \int_{s=0}^t \mu(s,x)\,u(s,x)\, \dd s\).

This leads to the independent study of the boundedness and Hölder regularity of
solutions \(w :=w(t,x)\) to the parabolic equation (with homogeneous Neumann
boundary conditions)
\begin{equation}
  \label{eq:rough-scalar-pde}
  a(t,x) \partial_t w(t,x) - \Delta w(t,x) = f(t,x)
  \qquad \text{assuming moreover that }  \quad \partial_t w \ge 0,
\end{equation}
where \(a\) is a ``rough'' coefficient, i.e., we only suppose that
\begin{equation}
  \label{eq:assumption-a}
  0 < a_0 \le a \le c_0 a_0 < \infty ,
\end{equation}
for some constants \(a_0\) and \(c_0 \ge 1\).  
\par 
For the forcing \(f\), we assume that it lies in a suitable Lebesgue space.

The studied parabolic equation~\eqref{eq:rough-scalar-pde} is a special case of more
general parabolic equations in non-divergence form, which write
\begin{equation}\label{eq:general-non-divergence}
  \partial_t w - \sum_{i,j} a^{ij} \partial_{ij} w = f.
\end{equation}
In this general case, \textcite{krylov-1976-sequences-convex-functions} proved a
parabolic version of the ABP maximum principle, stating that $\| w \|_{\fsL^{\infty}}$
can be estimated by $ \| f \|_{\fsL^{1+d}}$, where $d$ is the space dimension.

We also quote the subsequent works
\cite{krylov-safonov-1979,krylov-safonov-1981,
  krylov-1987-some,escauriaza-1992-krylov-tsos,krylov-safonov-1994,
  krylov-2009-controlled,mooney-2019-krylov-safonov,krylov-2021-stochastic-ld}
and in the related elliptic case \cite{krylov-2021-review}, where it is shown
that Hölder regularity (that is, $w \in \fsC^{\alpha}$ for some $\alpha >0$) can even be obtained
for Equation \eqref{eq:general-non-divergence}.  \medskip

We show in this paper that the additional information \(\partial_t w \ge 0\) allows a
direct comparison with the heat equation (see
\cite{fellner-morgan-tang-2020-global} and in a similar matter
\cite{caputo-goudon-vasseur-2019-solutions-c}), allowing to treat all the
details related to the (homogeneous Neumann) boundary condition, and to provide
explicit constants in the estimates, including a precise and sharp dependency on
the forcing term \(f\).  \medskip

Our main result for solutions to \eqref{eq:rough-scalar-pde} is the following.
\begin{thm}\label{thm:final-hoelder-regularity}
  We consider a bounded, $\fsC^2$ domain \(\Omega \subset \R^d\).  Set \(T>0\), and
  \(p, q \in [1, \infty]\) such that $\gamma := 2 - \frac2p - \frac{d}q > 0$.  Then there
  exists a constant \(\alpha >0\) only depending on \(\gamma, d,c_0\), and a constant
  \(C_*\) depending on \(p,q,d, \Omega, T, a_0,c_0\) such that for any Lipschitz
  initial data $w_\init$, forcing data \(f \in \fsL^{p}((0,T]; \fsL^q(\Omega))\) and a
  coefficient \(a := a(t,x)\) satisfying the bound \eqref{eq:assumption-a}, a
  solution \(w \ge 0\) of \eqref{eq:rough-scalar-pde} over \((0,T] \times \Omega\) with
  homogeneous Neumann boundary data ($\vec n \cdot \nabla_x w = 0$ on $ (0,T] \times \partial\Omega$),
  lies in \(\fsC^{\alpha}([0,T] \times {\overline{\Omega}} ) \).  Moreover, the following
  estimate holds (with the notation $x_+ := \max(x,0)$):
  \begin{multline*}
    \| w \|_{\fsC^\alpha([0,T] \times {\overline{\Omega}} ) }
    \le C_*
    \left(
      \| f_+ \|_{\fsL^p((0,T]; \fsL^{q}(\Omega))}
      + \| w_{\init} \|_{\mathrm{Lip}}
    \right)^{1-\alpha/\gamma}
    \left(
      \| f \|_{\fsL^p((0,T]; \fsL^{q}(\Omega))}
      +
      \| w_\init \|_{\mathrm{Lip}}
    \right)^{\alpha/\gamma}.
  \end{multline*}
\end{thm}

In the above estimate, we denote
$$  \| w_{\init} \|_{\mathrm{Lip}} := \| w_{\init} \|_{L^{\infty}( {{\Omega}} ) } +
 \sup_{ x,x' \in \Omega, \, x \neq x'}
 \, \frac{|w(t,x) - w(t',x')|}{ |x-x'|}\,  , $$  
$$  \| w \|_{\fsC^\alpha([0,T] \times {\overline{\Omega}} ) } := \| w \|_{L^{\infty}([0,T] \times {{\Omega}} ) } +
 \sup_{t,t' \in [0,T], x,x' \in \Omega: (t,x) \neq (t',x')}
 \, \frac{|w(t,x) - w(t',x')|}{|t - t'|^{\alpha/2} + |x-x'|^{\alpha}}\, ,$$
and $\fsC^\alpha([0,T] \times {\overline{\Omega}})$ is the space of functions $w$ from  $[0,T] \times {\Omega}$ to $\R$,
for which $ \| w \|_{\fsC^\alpha([0,T] \times {\overline{\Omega}} ) }$ is finite.
\medskip

The constant $\alpha$ can be taken as 
$$ \alpha := \min \left( \frac{ \ln ([1 - \delta]^{-1}) }{ \ln 4 } , \gamma, \frac12 \right), $$
where 
$$ \delta := {\frac{13}{1568}} \, \left( 98\,\pi \,\right)^{-\frac{d}2}\, d^{\frac{d}2+1}   \, \ee^{ - d c_0} \,  |B_d(0, 1)|. $$

\begin{remark}
  Note that we obtain the Hölder regularity uniform up to the boundary of the
  domain, and to arbitrary small times.  To treat those small times, we need to
  assume some regularity of the initial datum. In
  \cref{thm:final-hoelder-regularity}, we supposed that $w_{\init}$ has
  Lipschitz regularity, which is not optimal (but leads to simplified proofs).
\end{remark}

\begin{remark}
  Compared to the ABP maximum principle in which the forcing lies in $L^p$ with
  $p=1+d$, \cref{thm:final-hoelder-regularity} gets us close to the critical
  estimate when $p=q$, as we get arbitrary close to the critical exponent
  \(p=1+d/2\) for the forcing.
\end{remark}

The crucial step in the proof of our main Theorem is the decay of oscillations.
Hence we study around a point \(z_0 :=(t_0,x_0)\) the parabolic cylinder
\( Q^{\beta}_R(z_0) = (t_0-\beta R^2,t_0] \times B_d(x_0,R) \) where $R>0$ is a scaling,
parameter and \(\beta>0\) is a parameter which will be chosen small later on.  The
idea to obtain Hölder regularity is to show that the oscillation of $w$ over a
cylinder \(Q^{\beta}_{R/4}(z_0)\) is controlled by the oscillation over
\(Q^{\beta}_R(z_0)\) with a factor strictly less than one.  \medskip

The following Proposition shows such a decay where we include a possible
boundary with Neumann boundary data (Note that by scaling for
\cref{thm:final-hoelder-regularity} it will only be used for \(R=1\); we state
the general case \(R>0\) for possible other applications).

\begin{proposition}\label{thm:oscillation-decay-boundary}
  Fix $d \in \N - \{0\}$, \(p, q \in [1, \infty] \) such that
  $ \gamma := 2 - \frac{2}{p} - \frac{d}q > 0$.  There exist constants \(\beta,\delta,A>0\)
  depending only on $d$, $p$, $a_0$, $c_0$ such that for any $R>0$ and any
  function \(\phi : \R^{d-1} \to \R\) with \(\phi \le \frac{R}{11\,d}\) and
  \(\| \nabla \phi\|_{\infty} \le \frac1{11\,d}\) defining the domains
 \begin{equation*}
    \Omega_R = \{ (x',x_d) \in B_d(0,R) : x_d > \phi(x') \},
    \quad  \quad
    \Omega_{R/4} = \{ (x',x_d) \in B_d(0,R/4) : x_d > \phi(x') \},
  \end{equation*}
  then a function \(w\) with \(0\le w \le 1\) solving \eqref{eq:rough-scalar-pde}
  with a coefficient $a := a(t,x)$ satisfying the bound \eqref{eq:assumption-a}
  over \((-\beta R^2,0] \times \Omega_R\), with homogeneous Neumann boundary condition along
  the graph, that is
 \begin{equation*}
    \vec n \cdot \nabla_x w = 0
    \quad \text{on }
    (-\beta R^2,0] \times
    \{(x',x_d) \in B_d(0,R) : x_d = \phi(x') \}, 
  \end{equation*}
  also satisfies the reduction of oscillation bound
  \begin{equation*}
    \oscillation_{(-\beta R^2/16,0] \times \Omega_{R/4}} w  \le 1-\delta + C_{f,R} \,\, 
    \, R^{2 - \frac2{p} - \frac{d}{q}} \, \| f \|_{\fsL^p((- \beta  \,  R^2,0); \fsL^{q}( \Omega_R))},
  \end{equation*}
  where \(C_{f,R} := A \,C^*_{\frac{49 \,R^2}{200\, d}, d} \), the notation
  $C^*_{T, d}>0$ being defined in \cref{thm:heat-kernel-phi-bound}.

  The constants $\beta,\delta,A>0$ can be explicitly estimated as:
  \begin{equation*}
    \beta :=   \frac{49 \,a_0}{200\, d}; \quad
    \delta :=  {\frac{13}{1568}} \, \left( 98\,\pi \,\right)^{-\frac{d}2}\, d^{\frac{d}2+1}   \, \ee^{ - d c_0} \,  |B_d(0, 1)| ; \quad
    A :=
    (a_0c_0)^{\frac{d}{2q}} \,
    \left(1 - \frac{d p'}{2q} \right)^{- \frac1{p'}} \, \left( \frac{25}{64} \, \frac{a_0^2}{2d} \right)^{\frac1{p'} - \frac{d}{2q}} .
  \end{equation*}
\end{proposition}

\begin{remark}
  Note that we only impose a boundary condition on \(w\) along the graph of
  \(\phi\).  Further note that by taking \(\phi \equiv -R\), it also covers the case
  without boundary.
\end{remark}

\bigskip

In the proofs of the applications (\cref{proskt} and
\cref{pron}), we need  nonstandard interpolation
estimates that we propose to call ``one-sided interpolation estimates''. Since
they may be of interest for other results, we write them down here as a
self-contained Proposition:

\begin{proposition}\label{inpppg}
  Let $\Omega \subset \R^d$ be a bounded \(\fsC^2\) domain, and assume
  $p,q \in [1, \infty)$ satisfying \(\frac{d}{2} < \frac{3}{2}p \le q\).  Then for any
  $u,w : \Omega \to \R$ such that $0 \le u \le \Delta w$ in \(\Omega\), it holds that
  \begin{equation}\label{inppalk1}
    \|u \|_{\fsL^q(\Omega)} \le  C_{d,p,q}\, \| {w} \|_{\fsL^{\infty}(\Omega)}^{1-\frac{2-d/q}{3-d/p}}\,
    \| \nabla u\|_{\fsL^p(\Omega)}^{\frac{2-d/q}{3-d/p}}
    + C_{d,p,q}\, \| {w} \|_{\fsL^{\infty}(\Omega)} ,
  \end{equation}
  where $C_{d,p,q}>0$ is a constant depending only on $d,p,q$.

  Moreover, assuming that $\alpha \in (0,1)$, and that
  $q \ge \frac{3 - \alpha}{2 - \alpha}\, p > \frac{d}{2 - \alpha}$.  Then for any $u,w : \Omega \to \R$
  such that $0 \le u \le\Delta w$, it holds that
  \begin{equation}\label{inppalk}
    \|u \|_{\fsL^q(\Omega)}
    \le  C_{d,p,q, \alpha}\, \| {w} \|_{\fsC^{\alpha}(\overline{\Omega}) }^{1-\frac{2-\alpha - d/q}{3- \alpha - d/p}}\,
    \| \nabla u\|_{\fsL^p(\Omega)}^{\frac{2- \alpha - d/q}{3- \alpha - d/p}}
    +  C_{d,p,q, \alpha}\, \| {w} \|_{\fsC^{\alpha}(\overline{\Omega})},
  \end{equation}
  where $C_{d,p,q, \alpha}>0$ is a constant depending only on $d,p,q, \alpha$.
\end{proposition}

 Note that the name  ``one-sided interpolation estimates'' is related to the fact that we assume that $0 \le u \le \Delta w$ instead of assuming the stronger (and more classical) identity $u = \Delta w$.
\medskip

 A study of more general estimates of this form with applications to different equations is in preparation \cite{BDD2}.
\medskip

\bigskip

\Cref{hh} is devoted to the proof of \cref{thm:final-hoelder-regularity}.  We
start with estimates for the heat equation in
Subsection~\ref{sec:rough:heat-kernel}.  Then we use a comparison between the
heat equation and the equation with rough coefficients
\eqref{eq:rough-scalar-pde}, \eqref{eq:assumption-a}, and obtain the decay of
oscillation (\cref{thm:oscillation-decay-boundary}) in
Subsection~\ref{sec:rough:comparison}.  Finally, we perform global bounds to
start the iteration and conclude \cref{thm:final-hoelder-regularity} in
Subsection~\ref{sec:rough:conclusion} by using iteratively
\cref{thm:oscillation-decay-boundary}.  We present the proof of our first
application (\cref{proskt}) in Section~\ref{trois}, and of our second
application (\cref{pron} and its extensions, including \cref{pron2}) in
Section~\ref{quatre}.  The Appendix~\ref{AB} contains the proof of the
``one-sided interpolation estimates'' appearing in \cref{proskt}.

\section{Hölder regularity for heat equation with rough scalar coefficient}
\label{hh}

\subsection{Estimates for the heat kernel}
\label{sec:rough:heat-kernel}

We first establish estimates for the heat kernel, in particular in the case of a
domain with boundaries.  There is a huge literature on heat kernels in which
more advanced estimates in more general settings are presented, see for example
 \cite{saloff-coste-2010,gyrya-saloff-coste-2011-neumann-dirichlet, choulli-kayser-2015-gaussian-neumann-green,wang-1997-global} and the references therein.
  \medskip

The starting point for our estimates is the extension operator for functions
lying in a Sobolev space from a domain towards the whole space.  This is a
classical topic, see e.g.\
\cite{brezis-2011-functional-sobolev,evans-2010-partial}.  We recall a classical
estimate for a general domain, and state moreover the basic reflection estimate,
pointing out the dependency of the constant on the boundary.

\begin{lemma} \label{l6}
  Let $\Omega$ be a $\fsC^1$ domain of \,$\R^d$. Then, there exists an extension operator
  $E_2: \fsH^1(\Omega) \mapsto \fsH^1(\R^d)$ and a constant $C_E(\Omega)>0$ depending only on $\Omega$
  such that
  $$ \|E_2 u\|_{\fsH^1(\R^d)} \le C_E(\Omega) \,  \|u\|_{\fsH^1(\Omega)} . $$

  Moreover, for a $\fsC^1$ function \(\phi : \R^{d-1} \mapsto \R\), we define the domain
  \begin{equation*}
    \Omega_+ = \{ (x',x_d) \in \R^d : x_d > \phi(x') \}.
  \end{equation*}
  Then there exists an extension operator \(E_2 : \fsH^1(\Omega_+) \mapsto \fsH^1(\R^d)\) and
  $\fsL^1(\Omega_+) \mapsto \fsL^1(\R^d)$ such that
  \begin{equation*}
    \| E_2 u \|_{\fsH^1(\R^d)} \le 2 \,\sqrt{1 + \|\nabla \phi\|_{\infty}^2}\, \,
    \| u \|_{\fsH^1(\Omega_+)}.
  \end{equation*}  
  and
  \begin{equation*}
    \| E_2 u \|_{\fsL^1(\R^d)} \le 2 \,
    \| u \|_{\fsL^1(\Omega_+)}.
  \end{equation*} 
\end{lemma}

\begin{proof}
  The first part of the statement is a classical result,
  cf.~\cite{brezis-2011-functional-sobolev}.  For the second part, it suffices
  to define and bound the extension for a smooth function
  \(u \in \fsH^1(\Omega_+)\). We define the extended function \(\bar u\) by reflection
  as
  \begin{equation*}
    E_2 u(x',x_d) =   \bar{u}(x',x_{d}) :=
    \begin{cases}
      u(x',x_d) & \text{if } x_d \ge \phi(x'), \\
      u(x',2\phi(x')-x_d)
                &\text{if } x_d < \phi(x').
    \end{cases}
  \end{equation*}
  Then, we directly get
  \(\| \bar u \|_{\fsL^1(\R^d)}^2 \le 2\, \| u \|_{\fsL^1(\Omega_+)}^2\) and
  \( \| \bar u \|_{\fsL^2(\R^d)}^2 \le 2\, \| u \|_{\fsL^2(\Omega_+)}^2\).  Moreover, for
  \(x_d < \phi(x')\), we find for the derivatives
  \begin{align*}
    \nabla_i \bar u &= \nabla_i u|_{(x',2\phi(x')-x_d)}
                 + 2 \nabla_i \phi\, \nabla_d u|_{(x',2\phi(x')-x_d)},
                 \qquad i=1,\dots,d-1\\
    \nabla_d \bar u &= \nabla_d u|_{(x',2\phi(x')-x_d)}.
  \end{align*}
  This shows that
  \begin{equation*}
    \begin{split}
      \| \bar u \|_{\fsL^2(\R^d)}^2 +  \| \nabla \bar u \|_{\fsL^2(\R^d)}^2
      &\le \int_{\Omega_+} \bigg( 2\, |u|^2 + 2\,| \nabla_d u|^2 + \sum_{i=1}^{d-1} [ 3\,|\nabla_i u|^2 + 4 \,|\nabla_i \phi|^2\, |\nabla_d u|^2 ] \bigg)\, \dd x \\
      &\le (4 + 4 \|\nabla \phi\|_{\infty}^2) \,  \| \bar u \|_{\fsH^1(\Omega_+)}^2,
    \end{split}
  \end{equation*}
  which then yields the statement.
\end{proof}

Given a domain \(\Omega \subset \R^d\), we split its boundary into \(\partial\Omega_d\) (Dirichlet
boundary condition) and \(\partial\Omega_n\) (Neumann boundary condition). For the domain
with the split boundary, we define for a fixed \(y \in \Omega\) the mixed heat kernel
\(\Gamma_{\Omega}\) as the solution to
\begin{equation*}
  \left\{
    \begin{aligned}
      &(\partial_t - \Delta_x) \Gamma_\Omega(t,x,y) = 0
      & & \text{for } (t,x) \in (0,\infty) \times \Omega, \\
      & \Gamma_{\Omega}(t,x,y) = 0
      & & \text{for } (t,x) \in (0,T) \times \partial\Omega_d, \\
      & \vec n \cdot \nabla_x \Gamma_{\Omega}(t,x,y) = 0
      & & \text{for } (t,x) \in (0,T) \times \partial\Omega_n, \\
      & \lim_{t \downarrow 0} \Gamma_\Omega(t,\cdot,y) = \delta_y.
    \end{aligned}
  \right.
\end{equation*}

We now develop estimates on \(\Gamma_{\Omega}\).

\begin{lemma}\label{thm:heat-kernel-upper-bound}
  Using the notation above, we suppose that there exists \(E\) an extension
  operator from the piecewise-$\fsC^1$ domain \(\Omega\) to \(\R^d\) in $\fsH^1$ and
  $\fsL^1$ for functions \(\{u \in \fsH^1(\Omega) : u|_{\partial\Omega_d}=0\}\) and
    set \(|||E||| = \max(\| E \|_{\fsH^1 \to \fsH^1}, \| E \|_{\fsL^1 \to \fsL^1})\).  Then for any time
  \(T>0\), there exists a constant \(C_{\Omega, T, d}>0\) depending on $T$, $d$, and
  $\Omega$ through $ ||| E |||$ only (and increasing with respect to $T$), such that
  for all $p \in [1,\infty)$, \(y \in \Omega\) and \(t \in (0,T)\),
  \begin{equation*}
    \| \Gamma_{\Omega}(t, \cdot, y) \|_{\fsL^1(\Omega)} \le 1, \qquad
    \| \Gamma_{\Omega}(t, \cdot, y) \|_{\fsL^p(\Omega)} \le C_{\Omega, T, d}\, t^{-d/(2 p')},
    \qquad
    \| \Gamma_{\Omega}(t, \cdot, y) \|_{\fsL^\infty(\Omega)} \le C_{\Omega, T, d}\, t^{-d/2},
  \end{equation*}
  and for all \(x \in \Omega\) and \(t \in (0,T)\),
  \begin{equation*}
    \| \Gamma_{\Omega}(t, x, \cdot) \|_{\fsL^1(\Omega)} \le 1, \qquad
    \| \Gamma_{\Omega}(t, x, \cdot) \|_{\fsL^p(\Omega)} \le C_{\Omega, T, d}\, t^{-d/(2p')},
    \qquad
    \| \Gamma_{\Omega}(t, x, \cdot) \|_{\fsL^\infty(\Omega)} \le C_{\Omega, T, d}\, t^{-d/2}.
  \end{equation*}
\end{lemma}

\begin{proof}
  By the construction of the heat kernel, we have for all \(t>0\) that
  \(\Gamma_\Omega \ge 0\) (thanks to the minimum principle) and
  \(\|\Gamma_\Omega(t,\cdot,y)\|_{\fsL^1(\Omega)} \le 1\) (since \(\vec n \cdot \nabla \Gamma_{\Omega}|_{\partial\Omega_d} \le 0\),
  where $\vec n$ is the outward normal vector at a point of $\partial\Omega$).

  Moreover, we find the
  dissipation for the evolution of the square of the $L^2$ norm:
  \begin{equation*}
    \frac12\, \frac{\dd}{\dd t}
    \| \Gamma_{\Omega}(t,\cdot,y) \|_{\fsL^2(\Omega)}^2
    = - \| \nabla \Gamma_{\Omega}(t,\cdot,y) \|_{\fsL^2(\Omega)}^2.
  \end{equation*}
  The Nash inequality states that there exists a constant
  \(C_{\mathrm{Nash},d} >0\) only depending on $d$ such that for any function
  \(v\) over \(\R^d\),
  \begin{equation*}
    \| v \|_{\fsL^2(\R^{d})}^{1+\frac 2d}
    \le C_{\mathrm{Nash},d}\, \| \nabla v \|_{\fsL^2(\R^d)} \,
    \| v \|_{\fsL^1(\R^d)}^{ \frac2d }.
  \end{equation*}
  Using the extension operator \(E\), we end up with
  \begin{equation*}
    \begin{aligned}
      \| \Gamma_{\Omega} \|_{\fsL^2(\Omega)}^{1+\frac 2d}
      &\le  \|E \Gamma_{\Omega} \|_{\fsL^2(\R^d)}^{1+\frac 2d}  \\
      &\le C_{\mathrm{Nash},d}\,   \| \nabla(E \Gamma_{\Omega}) \|_{\fsL^2(\R^d)} \,  \|E \Gamma_{\Omega} \|_{\fsL^1(\R^d)}^{\frac2d} \\
      &\le C_{\mathrm{Nash},d}\,  ||| E |||^{1 + \frac2d}\, \sqrt{ \| \Gamma_{\Omega} \|_{\fsL^2(\Omega)}^2 +  \| \nabla \Gamma_{\Omega} \|_{\fsL^2(\Omega)}^2 } \,  \|\Gamma_{\Omega} \|_{\fsL^1(\Omega)}^{\frac2d} ,
    \end{aligned}
  \end{equation*}
  so that, using the estimate \(\|\Gamma_\Omega(t,\cdot,y)\|_{\fsL^1(\Omega)} \le 1\), we end up with
  \begin{equation*}
    \| \nabla \Gamma_{\Omega} \|_{\fsL^2(\Omega)}^2
    \ge C''\,    \| \Gamma_{\Omega} \|_{\fsL^2(\Omega)}^{2+\frac4d}  -  \| \Gamma_{\Omega} \|_{\fsL^2(\Omega)}^2 ,
  \end{equation*}
  with $C'' := C_{\mathrm{Nash},d}^{-2}\, |||E |||^{-2 - \frac4d}$, and finally
  \begin{equation*}
    \frac12\, \frac{\dd}{\dd t}
    \| \Gamma_{\Omega}(t,\cdot,y) \|_{\fsL^2(\Omega)}^2
    \le -C''\,
    \| \Gamma_{\Omega}(t,\cdot,y)  \|_{\fsL^2(\Omega)}^{ 2+\frac4d }  + \,   \| \Gamma_{\Omega}(t,\cdot,y)  \|_{\fsL^2(\Omega)}^{2} .
  \end{equation*} 
  Denoting $q(t) :=  \bigg( \| \Gamma_{\Omega}(t,\cdot,y)  \|_{\fsL^2(\Omega)}^{ 2} \bigg)^{- \frac2d}$, we see therefore that
  $$ q'(t) \ge \frac{4}{d} \, (C'' - q(t)), $$
  which implies
  $$ q(t) \ge C''\, \frac{4}{d}  \, t\, \ee^{- \frac4d\, t} , $$
  so that in the end
  $$ \| \Gamma_{\Omega}(t,\cdot,y)  \|_{\fsL^2(\Omega)}^{ 2} \le \left( C''\, \frac{4}{d}  \right)^{- \frac{d}2} \, \ee^{2T}\, t^{-\frac{d}2}. $$

  This yields the claimed \(\fsL^2\) bound. 
  \medskip

  For the \(\fsL^{\infty}\) bound, note first that the Laplace operator is self-adjoint, so
  that \(\Gamma_{\Omega}(t,x,y) = \Gamma_{\Omega}(t,y,x)\).  By the semigroup property, we therefore
  find that
  \begin{align*}
    \Gamma_{\Omega}(t,x,y) &= \int_{\Omega} \Gamma_{\Omega}(t/2,x,z) \, \Gamma_{\Omega}(t/2,z,y)\, \dd z\\
                 &\le \| \Gamma_{\Omega}(t/2,x,\cdot) \|_{\fsL^2(\Omega)}\; \| \Gamma_{\Omega}(t/2,\cdot,y) \|_{\fsL^2(\Omega)}\\
                 & \le \| \Gamma_{\Omega}(t/2,\cdot,x) \|_{\fsL^2(\Omega)}\; \| \Gamma_{\Omega}(t/2,\cdot,y) \|_{\fsL^2(\Omega)},
  \end{align*}
  and we use the already established \(\fsL^2\) bound to conclude. 
  \medskip

  A direct interpolation between the $\fsL^1$ and $\fsL^{\infty}$ estimates yields
  the $\fsL^p$ estimate stated in the Lemma.  Finally, the self-adjointness
  argument leads to the second group of estimates.
\end{proof}

In the context of the domains appearing in the proof of
\cref{thm:oscillation-decay-boundary} (but keeping the notations of Lemma \ref{thm:heat-kernel-upper-bound}), we write down for the related
geometry the following statement:

\begin{lemma}\label{thm:heat-kernel-phi-bound}
  Fix \(T>0\).  Then there exists a constant \(C^*_{T,d}\) such that for any domain
  \begin{equation*}
    \Omega := B_d(0,R) \cap \{ (x',x_d) \in \R^d : x_d > \phi(x') \}
  \end{equation*}
  for \(R>0\) and \(\phi : \R^{d-1} \to \R\) with \(\| \nabla \phi \|_{\infty} \le \frac{1}{11d}\) with
  \(\partial\Omega_n = B_d(0,R) \cap \{ x',x_d) \in \R^d : x_d = \phi(x') \}\) and
  \(\partial\Omega_d = \partial\Omega \setminus \partial\Omega_n\) the constant \(C_{\Omega,T,d}\) appearing in
  \cref{thm:heat-kernel-upper-bound} is bounded by \(C^*_{T,d}\).
\end{lemma}

\begin{proof}
  For the domain considered in the statement, we need an extension operator
  \(E\) for functions $f \in \fsH^1(\Omega)$ such that $f|_{\partial\Omega_d} = 0$, for the norms
  \(\fsH^1\) and \(\fsL^1\).

  We can first extend $f$ to \(\Omega_+ := \{ (x',x_d) \in \R^d : x_d > \phi(x') \}\) by
  setting $E_1f =0$ on \(\Omega_+\setminus \Omega\), and $E_1f =f$ on $\Omega$. Recalling that
  $ f|_{\partial\Omega_d} = 0$, we see that $||E_1||_{H^1 \to H^1} = 1$ and
  $||E_1||_{L^1 \to L^1} = 1$.  We then use \cref{l6} (with $\Omega$ in this lemma
  corresponding to $\Omega_+$ here) to build the operator $E_2$ from $H^1(\Omega_+)$ to
  $H^1(\R^d)$, so that still thanks to \cref{l6},
  $||| E_2\,E_1 ||| = ||| E_2 ||| \le 2 \sqrt{1+\frac{1}{121d^2}}$, and we
  conclude the Lemma by setting $E := E_2\,E_1$.
\end{proof}

We now write down an estimate which is a direct consequence of
\cref{thm:heat-kernel-upper-bound}, and which will be used several times in the
sequel.

\begin{lemma}\label{corl}
  Let $\Omega$ be a piecewise - $\fsC^1$ domain of $\R^d$, and suppose that there
  exists \(E\) an extension operator from \(\Omega\) to \(\R^d\) for $\fsH^1$ and
  $\fsL^1$ functions. We consider $f \in \fsL^p([0,T]; \fsL^q(\Omega))$, with
  $p,q \in [1,\infty]$ and $\frac{2}{p} + \frac{d}q < 2$.  Then, for $t\in [0,T]$,
  $x\in \Omega$, $a>0$,
  \begin{equation*}
    \bigg| \int_0^t \int_{\Omega}  \Gamma_{\Omega} (\frac{t-s}a, x,y) \, f(s,y)\, \dd y\, \dd s \bigg|
    \le C_{\Omega, \frac{T}a, d} \, a^{\frac{d}{2q}} \,
    \left( \frac{T^{1 - \frac{d}2\frac{p'}{q}}}{1 - \frac{d}2\frac{p'}{q}} \right)^{\frac1{p'}} \,
    \|f\|_{\fsL^p([0,T]; \fsL^q(\Omega))} .
  \end{equation*}
\end{lemma}

\begin{proof}
  We compute
  \begin{equation*}
    \begin{split}
      \bigg| \int_0^t \int_{\Omega}  \Gamma_{\Omega} (\frac{t-s}a, x,y) \, f(s,y)\, \dd y\, \dd s \bigg|
      &\le
        \int_0^t  \| \Gamma_{\Omega} (\frac{t-s}a, x, \cdot) \|_{\fsL^{q'}(\Omega)} \, \|f(s,\cdot\,)\|_{\fsL^{q}(\Omega)} \, \dd s \\
      &\le C_{\Omega, \frac{T}a, d}  \int_0^t \left( \frac{t-s}a \right)^{ - \frac{d}{2q} } \, \|f(s,\cdot\,)\|_{\fsL^{q}(\Omega)} \, \dd s \\
      &\le C_{\Omega, \frac{T}a, d} \, a^{\frac{d}{2q} } \,
        \bigg[  \int_0^t ( {t-s} )^{ - \frac{d p'}{2q} } \,\dd s\bigg]^{\frac1{p'}}
        \, \bigg( \int_0^t \, \|f(s,\cdot\,)\|_{\fsL^{q}(\Omega)}^p \, \dd s \bigg)^{\frac1p} \\
      &\le C_{\Omega, \frac{T}a, d} \, a^{\frac{d}{2q} } \,
  \left( \frac{T^{1 - \frac{d}2\frac{p'}{q}}}{1 - \frac{d}{q}\frac{p'}{2}} \right)^{\frac1{p'}}  \,
        \|f\|_{\fsL^p([0,T]; \fsL^q(\Omega))} .\qedhere
    \end{split}
  \end{equation*}
\end{proof}

For treating  small times, we will also need the following moment bound.

\begin{lemma}\label{thm:heat-kernel-moment-bound}
  Assume the setup of \cref{thm:heat-kernel-upper-bound}.  For \(1/2>\epsilon>0\),
  there exists a constant \(\tilde{C}_{\epsilon, T} >0\) depending on $\epsilon$, $T$, $d$ and
  $\Omega$ only through $|\Omega|$ and $ |||E |||$, and increasing with respect to $T$,
  such that when $t\in [0,T]$ and $y\in \Omega$,
  \begin{equation*}
    m(t,y) := \int_{\Omega} |x-y|\, \Gamma_{\Omega}(t,x,y)\, \dd x \le \tilde{C}_{\epsilon, T}\,t^{\frac 12 - \epsilon}.
  \end{equation*} 
\end{lemma}

\begin{remark}
  Using a more careful argument as in
 \textcite{nash-1958-continuity} allows to remove
  the \(\epsilon\)  in the inequality above.
\end{remark}

\begin{proof}
  We fix \(y\) and use for $t > 0$ the entropy-like functional
  \begin{equation*}
    H(t,y) = \int_{\Omega} \Gamma_{\Omega}(t,x,y) \ln (1 + \Gamma_{\Omega}(t,x,y))\, \dd x.
  \end{equation*}
  By the \(\fsL^{1}\) and \(\fsL^{\infty}\) bound obtained in \cref{thm:heat-kernel-upper-bound}, we have that
  \begin{equation*}
    H(t,y) \le   \bigg[ \int_{\Omega} \Gamma_{\Omega}(t,x,y)  \, \dd x  \bigg]\, \,\ln (1 + C_{\Omega, T,d}\, t^{- \frac{d}{2}})
    \le  \tilde{C}_{\epsilon, T}\, (1 +  t^{-2\epsilon}).
  \end{equation*}
  Its dissipation is
  \begin{equation*}
    \frac{\dd}{\dd t} H = - \int_{\Omega} |\nabla \Gamma_{\Omega}|^2 \, \left( \frac1{(1+\Gamma_{\Omega})^2} +  \frac1{1+\Gamma_{\Omega}} \right)\, \dd x  ,
  \end{equation*}
  since
  $$ \int_{\partial\Omega_d \cup \partial\Omega_n }  \left( \frac{\Gamma_{\Omega}}{1+ \Gamma_{\Omega}} + \ln ( 1 + \Gamma_{\Omega}) \right)\,
  \nabla \Gamma_{\Omega} \cdot \vec n(x)\, \dd \sigma(x) = 0 . $$
  Hence we find that
  \begin{equation*}
    \frac{\dd}{\dd t} \left(  t^{4\epsilon} \, H \right)
    = - t^{4\epsilon}  \int_{\Omega} |\nabla \Gamma_{\Omega}|^2 \, \left( \frac1{(1+\Gamma_{\Omega})^2} +  \frac1{1+\Gamma_{\Omega}} \right)\, \dd x    + 4\epsilon  t^{4\epsilon -1} \, H,
  \end{equation*}
  so that (changing the definition of $\tilde{C}_{\epsilon,T}$),
  \begin{equation*}
    \begin{split}
      {t}^{4\epsilon} \, H(  { t} )
      + \int_0^{ t } s^{4\epsilon}  \int_{\Omega} |\nabla \Gamma_{\Omega}|^2 \, \left( \frac1{(1+\Gamma_{\Omega})^2} +  \frac1{1+\Gamma_{\Omega}} \right) (s) \, \dd x\, \dd s
      &\le  \int_0^t 4\epsilon\, s^{4\epsilon -1} \, \tilde{C}_{\epsilon, T} \, (1 + s^{-2\epsilon})\, \dd s \\
      &\le  \tilde{C}_{\epsilon, T}\,   ({ t}^{4\epsilon} + t^{2\epsilon}),
    \end{split}
  \end{equation*}
  and finally (since $t^{4\epsilon} \le T^{2\epsilon}\, t^{2\epsilon}$, and again changing the definition
  of $\tilde{C}_{\epsilon,T}$),
  \begin{equation*}
    \int_0^{ t} s^{4\epsilon} \int_{\Omega} \frac{|\nabla \Gamma_{\Omega}|^2}{1+\Gamma_{\Omega}}(s)\, \dd x\, \dd s
    \le  \tilde{C}_{\epsilon, T}\,   t^{2\epsilon} .
  \end{equation*}

  We observe then that (remembering that \(\vec n \cdot \nabla \Gamma_{\Omega}|_{\partial\Omega_d} \le 0\))
  \begin{align*}
    \frac{\dd}{\dd t} m
    &= \int_{\Omega} |x-y|\; \Delta \Gamma_{\Omega}\,\dd x \\
    &
      \le - \int_{\Omega} \nabla |x-y| \cdot \nabla \Gamma_{\Omega}\,\dd x  \\
    &
      \le \left( \int_{\Omega} \frac{|\nabla \Gamma_{\Omega}|^2}{ 1 + \Gamma_{\Omega}}\, \dd x \right)^{1/2}
      \left( \int_{\Omega} (1 + \Gamma_{\Omega} ) \, \dd x \right)^{1/2},
  \end{align*}
  since $| \nabla |x-y| |^2  \le 1$.

  Integrating the time derivative yields then the claimed bound. Indeed, using
  Cauchy-Schwarz inequality, and recalling the $\fsL^1$ bound obtained in
  \cref{thm:heat-kernel-upper-bound},
  \begin{equation*}
    \begin{split}
      m({ t})
      &\le  (1 + |\Omega|)^{\frac12} \, \int_0^{ t} s^{2\epsilon} \bigg(\int  \frac{|\nabla \Gamma_{\Omega}|^2}{1 + \Gamma_{\Omega}}(s) \, \dd x \,\bigg)^{\frac12} \,  s^{-2\epsilon} \, \dd s \\
      &\le  \bigg[ (1 + |\Omega|)\, t^{2\epsilon} \bigg]^{\frac12}\,
        \bigg[   \int_0^{ t}   s^{-4\epsilon} \, \dd s \bigg]^{1/2}
        \le   \tilde{C}_{\epsilon, T}\, t^{\frac 12 - \epsilon},
    \end{split}
  \end{equation*}
  where we changed one more time the definition of $\tilde{C}_{\epsilon, T}$ in the last inequality.
\end{proof}

\subsection{Decay of oscillations}
\label{sec:rough:comparison}

For the decay of oscillation, we obtain a new domain $U$ by cutting out a ball.
This reduced domain can have a part \(\partial U_n\) from the original domain where we
still have the Neumann boundary data.  We then have a subdomain \(U_r\) on which
we want to obtain a reduced oscillation.  Note that this subdomain typically can
hit or be close to the boundary part \(\partial U\).  For the reduction of the
oscillation, we consider a further subdomain \(U_e\) for which we assume that
the heat kernel starting from \(U_e\) is strictly positive in \(U_r\) after a
suitable time with a uniform bound.  \medskip

In this general setting, we formulate an abstract lemma for the decay of
oscillation.


\begin{lemma}\label{thm:abstract-decay}
  We consider a piecewise - $\fsC^1$ domain \(U \subset \R^d\) whose boundary is
  partitioned into \(\partial U_d\) and \(\partial U_n\) 
($\partial U_n$ can be empty).  We also
  consider subdomains \(U_r \subset U\) and \(U_e \subset U\), for which the fundamental
  solution of the heat equation \(\Gamma_{U}\) (with Dirichlet boundary condition
  along \(\partial U_d\) and Neumann boundary condition along \(\partial U_n\)), satisfies, for
  some given $T>0$, $q \in (1,\infty]$, and constants \(C_1>0\), \(C_2 < \infty\), the
  conditions
  \begin{equation} \label{abinf}
    \inf_{y \in U_e,x\in U_r}
    \inf_{t \in \left[\frac{T}{2c_0a_0},\frac{T}{a_0}\right]}
    \Gamma_{U}(t,x,y) \ge
    C_1,
  \end{equation}
  and for all $t \in [0, \frac{T}{a_0}]$,
  \begin{equation} \label{absup}
    \sup_{x \in U_r}   \| \Gamma_{U}(t,x,\cdot) \|_{\fsL^{q'}(U)} \le C_2\, t^{- \frac{d}{2q} } .
  \end{equation}
  Then, any solution \(w \in [0,1]\) to \eqref{eq:rough-scalar-pde}, \eqref{eq:assumption-a} over
  \([-T,0] \times U\) 
   and boundary condition
  \begin{equation*}
    \vec n \cdot \nabla_x w(t,x) = 0
    \qquad \text{for } (t,x) \in (-T,0) \times \partial U_n ,
  \end{equation*}
  is estimated (for \(p \in [1,\infty]\) with $p,q \in [1,\infty]$ and
  $\frac{2}{p} + \frac{d}q < 2$) as
  \begin{equation*}
    \oscillation_{[-\frac{T}2, 0] \times U_r} w
    \le 1 - \frac{C_1}{4} |U_e|
    +  C_2\, (a_0c_0)^{\frac{d}{2q}}\, (1 - \frac{dp'}{2q})^{-\frac1{p'}} \, T^{\frac1{p'} - \frac{d}{2q}}\,
    \|f\|_{\fsL^{p}([-T,0]; \fsL^{q}( U)) }  .
  \end{equation*}
\end{lemma}

\begin{remark}
  Note that we do not assume a boundary condition for \(w\) along \(\partial U_d\).
\end{remark}

\begin{proof}
  We can distinguish two cases, whether
  \(|\{w(-T,\cdot) \ge \frac 12 \} \cap U_e| \ge \frac 12 | U_e|\) or
  \(|\{w(-T,\cdot) \le \frac 12 \} \cap U_e| \ge \frac 12 |U_e|\).

  In the first case \(|\{w(-T,\cdot) \ge \frac 12 \} \cap U_e| \ge \frac 12 |U_e|\), we take
  \begin{equation*}
    E = \{w(-T,\cdot) \ge \frac 12 \} \cap \Omega_e
  \end{equation*}
  and consider the smooth (constant coefficients) problem
  \begin{equation*}
    \left\{
      \begin{aligned}
        & (a_0c_0\partial_t - \Delta)v = f
        & & \text{for } (t,x) \in (-T,0) \times U, \\
        & v(t,x) = 0
        & & \text{for } (t,x) \in (-T,0) \times \partial U_d, \\
        & \vec n \cdot \nabla_x v(t,x) = 0
        & & \text{for } (t,x) \in (-T,0) \times \partial U_n, \\
        & v(-T,x) = \frac 12 \ind_{E}(x)
        & & \text{for } x \in U.
      \end{aligned}
    \right.
  \end{equation*}

  Then
  \begin{equation*}
    (c_0a_0\partial_t - \Delta)(w-v)  = (c_0a_0-a) \partial_t w \ge 0
  \end{equation*}
  so that \(w \ge v\) by minimum principle.
\medskip

  We can express (for $t \in [-T,0], x \in U$)
  \begin{equation*}
    v(t,x)
    = \int_{y \in U} \Gamma_U(\frac{t+T}{a_0c_0},x,y) \,\frac 12 \ind_E(y)\,
    \dd y
    + \int_{s=-T}^t \int_{y \in U}
    \Gamma_U(\frac{t-s}{a_0c_0},x,y) f(s,y)\, \dd y\, \dd s,
  \end{equation*}
  so that for \((t,x) \in (-T/2,0) \times U_r\), it holds that (working like in the
  proof of \cref{corl})
  \begin{equation*}
    \begin{aligned}
      v(t,x)
      &\ge \frac{C_1}{4} |U_e|
        - \int_{s=-T}^t  \|\Gamma_U(\frac{t-s}{a_0c_0},x,\cdot )\|_{\fsL^{q'}(U)} \,\|f(s,\cdot\,)\|_{\fsL^{q}(U)} \, \dd s \\
      &\ge \frac{C_1}{4} |U_e|
        - C_2 \, \int_{s=-T}^t  \left(\frac{t-s}{a_0c_0} \right)^{- \frac{d}{2q}}  \,\|f(s,\cdot\,)\|_{\fsL^{q}(U)} \, \dd s \\
      &\ge \frac{C_1}{4} |U_e|
        - C_2 \, (a_0c_0)^{\frac{d}{2q}} \bigg[\int_{s=-T}^t  (t-s)^{- \frac{d p'}{2q}} \, \dd s \bigg]^{\frac1{p'}} \, \bigg( \int_{s=-T}^t  \|f(s,\cdot\,)\|_{\fsL^{q}(U)}^p \, \dd s \bigg)^{\frac1p} \\
      &\ge \frac{C_1}{4} |U_e|
        - C_2 \, (a_0c_0)^{\frac{d}{2q}} \bigg(\frac{T^{1 - \frac{d p'}{2q}}}{1 - \frac{d p'}{2q}} \bigg)^{\frac1{p'}} \,  \|f\|_{\fsL^{p}([-T,0]; \fsL^{q}(U)) } ,
    \end{aligned}
  \end{equation*}
  which shows the oscillation decay in this case. 
 \medskip

 In the other case \(|\{w(-T,\cdot) \le \frac 12 \} \cap U_e| \ge \frac 12 |U_e|\), we take

  \begin{equation*}
    \tilde{E} = \{w(-T,\cdot) \le \frac 12 \} \cap U_e
  \end{equation*}
  and set $\tilde{w} := 1 - w$. Then we consider
  $\tilde{v} := \tilde{v}(t,x)$ on $[-T,0] \times U$ such that
  \begin{equation*}
    \left\{
      \begin{aligned}
        & (a_0\partial_t - \Delta)\tilde v = -f
        & & \text{for } (t,x) \in (-T,0) \times U, \\
        & \tilde v(t,x) = 0
        & & \text{for } (t,x) \in (-T,0) \times \partial U_d, \\
        & \vec n \cdot \nabla_x \tilde v(t,x) = 0
        & & \text{for } (t,x) \in (-T,0) \times \partial U_n, \\
        & \tilde v(-T,x) = \frac 12 \ind_{\tilde{E}}(x)
        & & \text{for } x \in U.
      \end{aligned}
    \right.
  \end{equation*}
  We observe that
  \begin{equation*}
    (a_0 \partial_t - \Delta)(\tilde w - \tilde v)
    = (a \partial_t - \Delta)\tilde w
    - (a_0 \partial_t - \Delta)\tilde v
    + (a_0-a) \partial_t \tilde w
    \ge 0,
  \end{equation*}
  so that \(\tilde w \ge \tilde v\).  Then for $t\in [-T,0]$ and $x\in U$,
  \begin{equation*}
    \tilde v(t,x)
    = \int_{y \in U} \Gamma_U(\frac{t+T}{a_0},x,y) \frac 12 \ind_{\tilde{E}}(y)\,
    \dd y
    -
    \int_{s=-T}^t \int_{y \in U}
    \Gamma_U(\frac{t-s}{a_0},x,y) f(s,y)\, \dd y\, \dd s
  \end{equation*}
  and the result follows as before.
\end{proof}

This allows to proceed with the proof of the decay of oscillations.

\begin{proof}[Proof of \cref{thm:oscillation-decay-boundary}]

  We take as in the statement of the Proposition:
  \begin{equation*}
    \Omega_R = \{ (x',x_d) \in B_d(0,R) : x_d > \phi(x') \}
    \quad \text{ and } \quad
    \Omega_{R/4} = \{ (x',x_d) \in B_d(0,R/4) : x_d > \phi(x') \},
  \end{equation*}
  and we choose as set from which we take the measure information
  \begin{equation*}
    U_e := B_d\left(\frac{R}{5} e_d, \frac{R}{10}\right).
  \end{equation*}

  We then define $U := \Omega_R$ and $U_r := \Omega_{R/4}$, which gives (when
  $ \sup \phi \le \frac{R}{11 d}$ and $\|\nabla\phi\|_{\infty} \le \frac{1}{11 d}$) the geometric
  setup, cf.\ \cref{fig:oscillation-decay-boundary}.  We intend to use
  \cref{thm:abstract-decay} with $U, U_e$ and $U_r$ defined above, and with
  $\partial U_d := \partial {U} \cap\partial B_d(0,R)$,
  $\partial U_n := \{ (x',x_d) \in B_d(0,R) : x_d = \phi(x') \}$, for $w$ the function taken
  in the statement of \cref{thm:oscillation-decay-boundary}.

  \begin{figure}
    \centering
    \begin{tikzpicture}[scale=1.5]
      \begin{scope}
        \draw[domain=-4:4,samples=150]
        plot (\x,{0.1*sin(100*\x)+0.05*cos(500*\x)});
        \path[clip,domain=-4:4,samples=150]
        plot (\x,{0.1*sin(100*\x)+0.05*cos(500*\x)})
        -- (4,4) -- (-4.0,4) -- cycle;
        \draw[fill=green!30!white] (0,0) circle (4);
        \draw[pattern=north west lines] (0,0) circle (1);
        \draw[pattern=north east lines] (0,0.8) circle (0.4);
      \end{scope}
      \draw (0,0.9) node[fill=green!30!white] {\(U_e\)};
      \draw (-0.5,0.25) node[fill=green!30!white] {\(U_r\)};
      \draw (0,2) node {\(U\)};
    \end{tikzpicture}
    \caption{Geometric setup for \cref{thm:oscillation-decay-boundary} (with the
      notations of \cref{thm:abstract-decay}).}
    \label{fig:oscillation-decay-boundary}
  \end{figure}
  \medskip

  The upper bound \eqref{absup} of the heat kernel in the assumptions of
  \cref{thm:abstract-decay} follows from \cref{thm:heat-kernel-phi-bound}, with
  $C_2 := C_{\phi, T, d}$.

  We obtain the lower bound \eqref{abinf} of the heat kernel in the assumptions
  of \cref{thm:abstract-decay} thanks to a comparison with the heat kernel over
  the whole space. We recall first that the Green function of the heat equation
  on the whole space is $\Phi_{\R^d}(t,|x-y|)$, where (for $t>0$ and $x\in\R^d$)
  \begin{equation*}
    \Phi_{\R^d}(t,|x|)
    := (4\pi t)^{-d/2} \exp \left( - \frac{|x|^2}{4t} \right).
  \end{equation*}
  We observe that $\partial_t  \Phi_{\R^d}(t,|x|) \ge 0$ if and only if $t \le \frac{|x|^2}{2d}$, so that
  \begin{equation}  \label{ttr}
    \sup_{s \in [0,t]}  \Phi_{\R^d}(s,|x|) =  \Phi_{\R^d}(t,|x|) \,\ind_{t \le \frac{|x|^2}{2d}}
    +  \Phi_{\R^d}( \frac{|x|^2}{2d} ,|x|) \,\ind_{t \ge \frac{|x|^2}{2d}},
  \end{equation}
  \begin{equation}  \label{ttr2}
    \partial_t\left (   \sup_{s \in [0,t]}  \Phi_{\R^d}(s,x) \right) = \partial_t  \Phi_{\R^d}(t,x) \,1_{t \le \frac{|x|^2}{2d}} \ge 0 .
  \end{equation}

  We then claim that for  $t \ge 0$, \(y \in U_e\) and $x \in U$,
  \begin{equation} \label{cp1}
    \Gamma_{U}(t,x,y)
    \ge \Psi(t,x,y) := \Phi_{\R^d}(t,|x-y|) - \sup_{s \in [0,t]} \Phi_{\R^d}(s,\frac{7R}{10}).
  \end{equation}
  Indeed it is sufficient to use the comparison principle for $\Gamma_{U}$ and $\Psi$.

  Fix \(y \in U_e\) and observe for $(t,x) \in \R_+ \times U$ that
  \begin{equation*}
    (\partial_t - \Delta_x) ( \Gamma_{U} -  \Psi) = \partial_t \left(  \sup_{s \in [0,t]} \Phi_{\R^d}(s,\frac{7R}{10})  \right)  \ge 0 .
  \end{equation*}
  For \(x\) belonging to the Dirichlet boundary $\partial U_d$, we see that
  $(\Gamma_{U} - \Psi)(t,x,y) = - \Phi_{\R^d}(t,|x-y|) + \Phi_{\R^d}(t,\frac{7R}{10}) \ge 0$ since
  $|x - y| \ge \frac{7R}{10}$ when $x \in \partial U_d$ and $y \in U_e$.

  When $x$ belongs to the Neumann boundary $\partial U_n$, we see that
  $ \nabla (\Gamma_{U} - \Psi)(t,x,y) \cdot \vec n(x) = {- \nabla \Phi_{\R^d}(t,|x-y|) \cdot \vec n(x)} = \frac{x-y}{2t} \cdot \vec n(x) \, \Phi_{\R^d}(t, |x-y|)$. This
  last quantity is nonnegative as soon as $(x-y) \cdot \vec n(x) \ge 0$ for $x\in \partial U_n$,
  $y \in U_e$, which can be rewritten as
  $y_d - x_d + \sum_{i=1}^{d-1} (x_i - y_i) \, \partial_i \phi(x',x_d) \ge 0$ (still for
  $x\in \partial U_n$, $y \in U_e$).
 But for those $x,y$, we have
  \begin{equation*}
    \begin{split}
      y_d - x_d + \sum_{i=1}^{d-1} (x_i - y_i) \, \partial_i \phi(x',x_d)
      &\ge \frac{R}{10} - \sup \phi  - (d- 1) \|\nabla \phi\|_{\infty} \, \left(R +\frac{R}{10} \right) \\
      &= \bigg[\frac{1}{10} - \left( \frac{\sup \phi}R + \frac{11}{10}\,(d-1)\,\|\nabla \phi\|_{\infty}\right) \bigg]\, R,
    \end{split}
  \end{equation*}
  so that $ \nabla (\Gamma_{U} - \Psi)(t,x,y) \cdot \vec n(x) \ge 0$ as soon as
  $ \max(\frac{\sup \phi}R , \|\nabla \phi\|_{\infty} )\le \frac1{11\,d}$ (which is the assumption
  made in the statement of the Proposition).

  The comparison principle ensures then that \eqref{cp1} holds for $t \ge 0$,
  \(y \in U_e\) and $x \in U$.\medskip

  We see that when $t \le \frac{49}{200}\, \frac{a_0}{d}\, R^2$, then
  \(\sup_{s \in [0,t]} \Phi_{\R^d}(s,\frac{7R}{10}) = \Phi_{\R^d}(t,\frac{7R}{10})\),
  thanks to \eqref{ttr}.

  Moreover, since the distance between any points in \(U_e\) and \(U_r\) is
  always less than
  \begin{equation*}
    \sqrt{ \left(\frac{R}4 + \frac{3R}{10} \right)^2 + \left(\frac{R}4 + \frac{R}{10} \right)^2}
    \le \frac{\sqrt{170}}{20} \, R,
  \end{equation*}
  we get that for $t \le \frac{(7R/10)^2}{2d}$,  \(y \in U_e\) and $x \in U_r$,
  \begin{align*}
    \Gamma_{U}(t,x,y)
    &\ge \Phi_{\R^d}(t,|x-y|) -  \Phi_{\R^d}(t,\frac{7R}{10})\\
    &\ge \Phi_{\R^d}(t, \frac{\sqrt{170}}{20}\,R ) -  \Phi_{\R^d}(t,\frac{7}{10}\, R)  \\
    &\ge (4\pi t)^{-\frac{d}2} \, \ee^{ - \left( \frac{7}{10} \right)^2 \frac{R^2}{4t} } \, \bigg[ \ee^{ \frac{13}{200} \frac{R^2}{4t} } - 1 \bigg] .
  \end{align*}
  Then, for any $T>0$, \(y \in U_e\) and $x \in U_r$,
  \begin{align*}
    \inf_{t \in [\frac{T}{2 c_0 a_0}, \frac{T}{a_0}] \cap [0, \frac{(7R/10)^2}{2d}] }  \Gamma_{U}(t,x,y)
    &\ge
      \left(\frac{a_0}{4\pi T} \right)^{\frac{d}2}
      \, \ee^{ - \left( \frac{7}{10} \right)^2 \frac{c_0 a_0 R^2}{2T} } \,
      \bigg[ \ee^{ \frac{13}{200} \, \frac{a_0R^2}{4T} } - 1 \bigg]  \\
    &\ge
      {\frac{13}{200}}  \,  \left(\frac{a_0}{4\pi T} \right)^{\frac{d}2}
      \, \ee^{ - \left( \frac{7}{10} \right)^2 \frac{c_0 a_0 R^2}{2T} } \,  \frac{a_0 R^2}{4T} .
  \end{align*}
  Selecting now $T := a_0 \,  \frac{(7R/10)^2}{2d}$, we see that
  \begin{align*}
    \inf_{t \in [\frac{T}{2 c_0 a_0}, \frac{T}{a_0}] }  \Gamma_{U}(t,x,y)
    &\ge \, {\frac{13}{200}}
      \left(\frac{2d}{4\pi} \frac{100}{49} \frac1{R^2}\right)^{\frac{d}2}  \, \ee^{ - d c_0} \, 2d\, \frac{25}{49}
    \\
    &\ge  {\frac{13}{392}} \, \left(\frac{50}{49\pi} \right)^{\frac{d}2}\, d^{\frac{d}2+1}   \, \ee^{ - d c_0} \,  R^{-d}.
  \end{align*}

  We can now use \cref{thm:abstract-decay} with
  ($T := a_0 \, \frac{(7R/10)^2}{2d}$)
  \begin{equation*}
    C_1 :=  {\frac{13}{392}}   \,
    \left(\frac{50}{49\pi} \right)^{\frac{d}2}\, d^{\frac{d}2+1}   \, \ee^{ - d c_0} \,  R^{-d}, \qquad
    C_2 :=  C_{\phi, \frac{T}{a_0}, d},
  \end{equation*}
  and deduce that (with the notations of \cref{thm:abstract-decay})
  \begin{equation*}
    \oscillation_{[-\frac{T}2, 0] \times U_r} w
    \le 1 - \frac{C_1}{4} |U_e|
    +  C_2
    \, (a_0c_0)^{\frac{d}{2q}} \bigg(\frac{T^{1 - \frac{d p'}{2q}}}{1 - \frac{d p'}{2q}} \bigg)^{\frac1{p'}} \,  \|f\|_{\fsL^{p}([-T,0]; \fsL^{q}( U)) } ,
  \end{equation*}
  which rewrites
  \begin{multline*}
    \oscillation_{[- a_0 \,  \frac{(7R/10)^2}{4d}, 0] \times \Omega_{R/4}} w
    \le 1 - {\frac{13}{1568}} \, \left(\frac{50}{49\pi} \right)^{\frac{d}2}\, d^{\frac{d}2+1}   \, \ee^{ - d c_0} \,  |B_d(0, \frac{1}{10})| \\
    +\,  (a_0c_0)^{\frac{d}{2q}} \,   C_{\phi,\frac{49 \,R^2}{200\, d}, d}\,
    (1 - \frac{d p'}{2q})^{- \frac1{p'}} \, \left( \frac{49}{100} \, \frac{a_0}{2d} \right)^{\frac1{p'} - \frac{d}{2q}}\, R^{\frac2{p'} - \frac{d}{q}} \, \| f \|_{\fsL^p((-a_0 \,  \frac{(7R/10)^2}{2d},0); \fsL^{q}( \Omega_R))}.
  \end{multline*}
  This last inequality yields the statement of
  \cref{thm:oscillation-decay-boundary} with
  \begin{gather*}
    \beta :=  \frac{49}{200} \, \frac{a_0}d ; \qquad \delta :=  {\frac{13}{1568}} \, \left( 98\, \pi  \right)^{-\frac{d}2}\, d^{\frac{d}2+1}   \, \ee^{ - d c_0} \,  |B_d(0, 1)| , \\
    A :=   (a_0c_0)^{\frac{d}{2q}} \,
    \left(1 - \frac{d p'}{2q} \right)^{- \frac1{p'}} \, \left( \frac{49}{100} \, \frac{a_0}{2d} \right)^{\frac1{p'} - \frac{d}{2q}} .\qedhere
  \end{gather*}
\end{proof}

\subsection{Conclusion of Hölder regularity}
\label{sec:rough:conclusion}

To start the iteration of the decay of oscillation, we perform the following
estimates over the whole domain.

\begin{proposition}\label{thm:supremum-bound}
  We suppose that \(\Omega \subset \R^d\) is a \(\fsC^1\) domain, and over \((0,T] \times \Omega\), we
  consider a solution \(w\ge0\) of \eqref{eq:rough-scalar-pde},
  \eqref{eq:assumption-a} with initial data \(w_\init\) and homogeneous Neumann boundary
  data $\nabla w \cdot \vec n =0$ on $[0,T] \times \partial\Omega$.  Then for \(p, q \in [1, \infty] \) such that $ \gamma := 2 - \frac{2}{p} - \frac{d}q >0$,
  \begin{equation*}
    \| w \|_{\fsL^\infty((0,T] \times \Omega)}
    \le K_1\,
    C_{\Omega, \frac{T}{a_0}, d}\,\,
    T^{1 - \frac1{p} -\frac{d}{2q}}\,
    \| f_{+} \|_{\fsL^p((0,T] ; \fsL^q( \Omega))}
    + \| w_\init \|_{\fsL^\infty(\Omega)},
  \end{equation*}
  where \(K_1\) is a constant only depending on \(a_0,d,p\), and $C_{\Omega, T, d}>0$
  is a notation introduced in \cref{thm:heat-kernel-upper-bound}.

  Moreover, when $w_{\init}$ is Lipschitz, $R>0$, $\beta>0$ and $x_0\in\Omega$, one has for
  all $\epsilon \in (0,\frac12)$,
  \begin{equation}\label{eq:initial-oscillation}
    \oscillation_{(0, \beta\,R^2] \times B_d(x_0,R) \cap (0,T] \times \Omega} w
    \le K_2\, R^{1 - 2\epsilon} \| w_\init \|_{\mathrm{Lip}} + K_3\,
    R^{2 - \frac2{p} -\frac{d}{q}}\,
    \| f \|_{\fsL^p((0,T] ; \fsL^q( \Omega))} , 
  \end{equation} 
  where $K_2>0$ and $K_3>0$ only depend on $a_0, c_0, \epsilon, T, \beta, d, p$ and $\Omega$
  only through $|\Omega|$ and $|||E|||$ defined in
  \cref{thm:heat-kernel-upper-bound}.
\end{proposition}

\begin{proof}[Proof of \cref{thm:supremum-bound}]
  For the supremum bound, we start with the upper bound for \(w\) by comparing
  it to the solution \(v\) of the heat equation
  \begin{equation}
    \label{eq:supremum-bound:comparision-problem}
    \left\{
      \begin{aligned}
        &(a_0 \partial_t - \Delta) v(t,x) = f(t,x)
        & & \text{for } (t,x) \in (0,T] \times \Omega,\\
        & \vec n \cdot \nabla_x v(t,x) = 0
        & & \text{for } (t,x) \in (0,T] \times \partial\Omega,\\
        & v(0,x) = w_{\init}
        & & \text{for } x \in \Omega.
      \end{aligned}
    \right.
  \end{equation}
  In the interior \((0,T] \times \Omega\), we find that
  \begin{equation*}
    (a_0 \partial_t - \Delta)(v-w)
    = f - f + (a-a_0) \partial_t w
    \ge 0,
  \end{equation*}
  while the boundary condition and the initial data are identical for $v$ and
  $w$, so that by the comparison principle, \(w \le v\).

  We know that ($\Gamma_{\Omega}$ being the Green function associated to
  the Neumann boundary condition, that is $\partial\Omega_d = \emptyset$), when $t\in (0,T]$, $x\in \Omega$:
  \begin{equation*}
    v(t,x) = \int_{y\in\Omega} \Gamma_{\Omega}(\frac{t}{a_0},x,y)\, w_\init(y)\, \dd y
    + \int_{s=0}^t \int_{y\in\Omega} \Gamma_{\Omega}(\frac{t-s}{a_0},x,y)\, f(s,y)\,
    \dd y\, \dd s.
  \end{equation*}
  Then the $\fsL^1$ and $\fsL^{\infty}$ bound on the fundamental solution
  (\cref{thm:heat-kernel-upper-bound}) prove the estimate since
  $$ \int_{y\in\Omega} \Gamma_{\Omega}(\frac{t}{a_0},x,y)\, w_\init(y)\, \dd y \le \|w_\init\|_{\infty},$$
  and, thanks to \cref{corl},
  \begin{align*}
    \int_{s=0}^t \int_{y\in\Omega}
    \Gamma_{\Omega}(\frac{t-s}{a_0},x,y)\, f(s,y)\, \dd y\, \dd s
    &\le
      C_{\Omega, \frac{T}{a_0}, d} \, a_0^{\frac{d}{2q}} \,
      \left( \frac{T^{1 - \frac{d}2\frac{p'}{q}}}{1 - \frac{d}2\frac{p'}{q}} \right)^{\frac1{p'}} \,
      \|f_{+}\|_{\fsL^p([0,T]; \fsL^q(\Omega))},
  \end{align*}
  where we only need to bound the positive part \(f_+\) of \(f\) as \(\Gamma_{\Omega}\) is
  nonnegative.  Hence we obtain the supremum bound with
  $K_1 := \frac{a_0^{\frac{d}{2q}} }{(1 - \frac{d}{2} \frac{p'}{q})^{1 - \frac1p}}$.

  For the second bound \eqref{eq:initial-oscillation}, we consider $\tilde{v}$ defined by
  \begin{equation}
    \left\{
      \begin{aligned}
        &(a_0\,c_0 \partial_t - \Delta) \tilde{v}(t,x) = f(t,x)
        & & \text{for } (t,x) \in (0,T] \times \Omega,\\
        & \vec n \cdot \nabla_x \tilde{v}(t,x) = 0
        & & \text{for } (t,x) \in (0,T] \times \partial\Omega,\\
        &  \tilde{v}(0,x) = w_{\init}
        & & \text{for } x \in \Omega.
      \end{aligned}
    \right.
  \end{equation}
  We have in the interior \((0,T] \times \Omega\) the estimate
  \begin{equation*}
    (a_0\,c_0 \partial_t - \Delta)(w - \tilde{v}) = f - f + (a_0\,c_0 - a)\, \partial_t w
    \ge 0,
  \end{equation*}
  and identical boundary condition and initial data for $w$ and $\tilde{v}$.
  Then, by the comparison principle, we get \(w \ge \tilde{v}\), and finally
  (using what we already know from the proof of the supremum bound),
  \(v \ge w \ge \tilde{v}\), that we can rewrite ($\Gamma_{\Omega}$ being the Green function
  associated to the Neumann boundary condition, that is $\partial\Omega_d = \emptyset$), when
  $t\in (0,T]$, $x\in \Omega$, as:
  \begin{align*}
    &\int_{y\in\Omega} \Gamma_{\Omega}(\frac{t}{c_0a_0},x,y)\, w_\init(y)\, \dd y
      + \int_{s=0}^t \int_{y\in\Omega} \Gamma_{\Omega}(\frac{t-s}{c_0a_0},x,y)\, f(s,y)\,
      \dd y\, \dd s\\
    &\le
      w(t,x)\\
    &\le \int_{y\in\Omega} \Gamma_{\Omega}(\frac{t}{a_0},x,y)\, w_\init(y)\, \dd y
      + \int_{s=0}^t \int_{y\in\Omega} \Gamma_{\Omega}(\frac{t-s}{a_0},x,y)\, f(s,y)\,
      \dd y\, \dd s.
  \end{align*}

  From the moment bound of \cref{thm:heat-kernel-moment-bound}, if $w_\init$ is
  Lipschitz, we get (remembering that $ \Gamma_{\Omega}$ is the Green function associated
  to the Neumann boundary condition, so that $\int_{y\in\Omega} \Gamma_{\Omega}(t,x,y) \, \dd y =1$
  for any $t>0$, $x\in \Omega$), whenever $\epsilon \in (0, \frac12)$,
  \begin{multline}\label{jkj}
    \bigg|  \int_{y\in\Omega} \Gamma_{\Omega}(\frac{t}{a_0},x,y)\, w_\init(y)  \,  \dd y - w_\init (x) \bigg|
    =  \bigg| \int_{y\in\Omega} \Gamma_{\Omega}(\frac{t}{a_0},x,y)\,( w_\init(y) -  w_\init(x))\, \dd y \bigg| \\
    \le \|w_\init\|_{\mathrm{Lip}}\,  \int_{y\in\Omega} \Gamma_{\Omega}(\frac{t}{a_0},x,y)\,|x-y| \,  \dd y  \le \tilde{C}_{\epsilon, \frac{T}{a_0}}\, \|w_\init\|_{\mathrm{Lip}}\, \left(\frac{t}{a_0}\right)^{1/2 - \epsilon} .
  \end{multline}
  Then, using estimates \eqref{jkj} and \cref{corl}, we see that for
  $x,x' \in \Omega$,
  \begin{align*}
    &w(t,x) -  w_\init (x') = ( w(t,x) -  w_\init (x) ) + (w_\init (x) - w_\init (x')) \\
    &\le  \int_{y\in\Omega} \Gamma_{\Omega}(\frac{t}{a_0},x,y)\, w_\init(y)  \,  \dd y - w_\init (x)
      + |w_\init (x) - w_\init (x')| + \int_{s=0}^t \int_{y\in\Omega} \Gamma_{\Omega}(\frac{t-s}{a_0},x,y)\, |f(s,y)|\,
      \dd y\, \dd s \\
    &\le \|w_\init\|_{\mathrm{Lip}} \,\bigg(\tilde{C}_{\epsilon, \frac{T}{a_0}}\, \left(\frac{t}{a_0}\right)^{1/2 - \epsilon} + \, |x - x'|
      \bigg) +
      C_{\Omega, \frac{T}{a_0}, d} \, a_0^{\frac{d}{2q}} \, \left( \frac{t^{1 - \frac{d}2\frac{p'}{q}}}{1 - \frac{d}2\frac{p'}{q}} \right)^{\frac1{p'}} \,  \|f\|_{\fsL^p([0,T]; \fsL^q(\Omega))} .
  \end{align*}
  In the same way,
  \begin{align*}
    \bigg|  \int_{y\in\Omega} \Gamma_{\Omega}(\frac{t}{c_0 a_0},x,y)\, w_\init(y)  \,  \dd y - w_\init (x) \bigg|
    &\le \|w_\init\|_{\mathrm{Lip}}\,  \int_{y\in\Omega} \Gamma_{\Omega}(\frac{t}{c_0 a_0},x,y)\,|x-y| \,  \dd y \\
    &\le
      \tilde{C}_{\epsilon, \frac{T}{c_0 a_0}}\, \|w_\init\|_{\mathrm{Lip}}\, \left(\frac{t}{c_0 a_0}\right)^{1/2 - \epsilon} ,
  \end{align*}
  and finally
  \begin{multline}
    |w(t,x) -  w_\init (x') | \le
    \|w_\init\|_{\mathrm{Lip}} \, \bigg(\tilde{C}_{\epsilon, \frac{T}{a_0}}\, \left(\frac{t}{a_0}\right)^{1/2 - \epsilon} + \, |x - x'|   \bigg) \\
    +\, C_{\Omega, \frac{T}{a_0}, d} \, (a_0c_0)^{\frac{d}{2q}} \, \left( \frac{t^{1 - \frac{d}2\frac{p'}{q}}}{1 - \frac{d}2\frac{p'}{q}} \right)^{\frac1{p'}} \,  \|f\|_{\fsL^p([0,T]; \fsL^q(\Omega))} .
  \end{multline}
  Taking now  $R>0$, $\beta>0$, $x_0\in \Omega$, we get when $x,x' \in B(x_0, R)$, $0 \le t, t'\le \min(T, \beta R^2)$:
  \begin{multline}
    |w(t,x) -  w(t',x')| \le  2\, \tilde{C}_{\epsilon, \frac{T}{a_0}}\, \left(\frac{\beta}{a_0}\right)^{1/2 - \epsilon} \, \|w_\init\|_{\mathrm{Lip}} \,R^{1 - 2\epsilon} +   \|w_\init\|_{\mathrm{Lip}}\, R\\
    +\, 2\, C_{\Omega, \frac{T}{a_0}, d} \, (a_0c_0)^{\frac{d}{2q}} \, \left(1 - \frac{d}2\frac{p'}{q} \right)^{-\frac1{p'}} \,
    \beta^{\frac1{p'}  - \frac{d}{2q}} \,  \,R^{\frac2{p'}  - \frac{d}{q}} \,  \|f\|_{\fsL^p([0,T]; \fsL^q(\Omega))} ,
  \end{multline}
  which gives \eqref{eq:initial-oscillation} with
  \begin{equation*}
    K_2 := \max \bigg[2 \left(\frac{T}{\beta}\right)^{\epsilon} ,  2\, \tilde{C}_{\epsilon, \frac{T}{a_0}}\, \left(\frac{\beta}{a_0}\right)^{1/2 - \epsilon} \bigg], \qquad K_3 := 2\, \, C_{\Omega, \frac{T}{a_0}, d} \, (a_0c_0)^{\frac{d}{2q}} \, \left(1 - \frac{d}2\frac{p'}{q} \right)^{-\frac1{p'}} \,
    \beta^{\frac1{p'}  - \frac{d}{2q}}. \qedhere
  \end{equation*}
\end{proof}

%

We can finally prove our main Theorem.

\begin{proof}[Proof of \cref{thm:final-hoelder-regularity}]
  By the regularity of \(\Omega\), we can find \(R_0 >0\) such that for all \(0<R<R_0\)
  and points \(x_0 \in \Omega\), in a suitable rotated coordinate system,
  \begin{equation*}
    \Omega \cap B_d(x_0,R)
    = \{x_0 + R(x',x_d) : x_d > \phi(x')\} \cap B_d(x_0,R),
  \end{equation*}
  for a $\fsC^1$ function \(\phi : \R^{d-1} \to \R\) (depending on $R$) with
  \(\phi \le \frac{1}{11\,d}\) and \(\|\nabla \phi\|_{\infty} \le \frac1{11\,d}\).
  \medskip

  Indeed, one starts by observing that there exist some constants $K_0, L_0>0$
  such that at each point $y_0 \in \partial\Omega$ the domain can be expresssed as
  $\Omega \cap B(y_0, L_0) = \{ y_0 + y' - y_d\,\vec n(y_0)\in \R^{d}, y_d \ge \psi (y'),\, y' \cdot \vec n(y_0) = 0, y_d \in \R\} \cap B(y_0, L_0)$,
  where $\psi$ is a ($y_0$-dependent) $C^2$ function with $\psi(0) = 0$, $\nabla \psi (0) =0$,
  $\|\nabla^2\psi\|_{\infty} \le K_0$.  As a consequence, $\|\nabla\psi\|_{\infty} \le K_0\, \rho$ and
  $\|\psi\|_{\infty} \le K_0\, \rho^2/2$ on $B_{d-1}(0,\rho)$.

  Then, we take $R_0 := (11\, d\, K_0)^{-1}$. For $x_0 \in \Omega$ such that
  $d(x_0, \partial\Omega) \le R_0$ and $R \in (0,R_0]$, we define
  $\phi(x') = R^{-1} \, [\, \psi(R\,x') - |x_0 -y_0| \, ]$, where $y_0 \in \partial \Omega$ is such that
  $y_0 - x_0 // n(y_0)$, $|y_0 - x_0 |\le R_0$.  For $R \le R_0$, and $|x'| \le 1$, we
  see that $ |\nabla \phi (x')| \le K_0\, R \le (11\,d)^{-1} $, and
  $ \phi (x') \le (22\,d)^{-1} $.  \medskip

  For a given \(t_0 \in [\beta R^2, T]\) (where $\beta := \frac{49}{200}\, \frac{a_0}d$
  from \cref{thm:oscillation-decay-boundary}), we then consider
  \begin{equation*}
    \tilde w(\tilde t,\tilde x) = w(t_0+R^2\tilde t,x_0+R \tilde x) .
  \end{equation*}
  Then \(\tilde w\) is defined on $[-\beta, 0] \times \Omega_1$ with \(\Omega_1 := B_d(0,1) \cap \{(x',x_d) : x_d > \phi(x')\}\), and satisfies on this set
  \begin{equation*}
    \tilde a \partial_{\tilde t} \tilde w - \Delta \tilde w = R^2 \tilde f, \qquad \partial_{\tilde t} \tilde{w} \ge 0,
  \end{equation*}
  where \( \tilde a(\tilde t,\tilde x) := a(t_0+R^2\tilde t,x_0+R \tilde x) \)
  and \( \tilde f(\tilde t,\tilde x) := f(t_0+R^2\tilde t,x_0+R \tilde x)\).  We
  also see that $\tilde{w}$ satisfies the Neumann condition
  $\nabla_{\tilde x} \tilde{w} \cdot \vec n = 0$ on
  $(-\beta, 0) \times \{(x',x_d) \in B_d(0,1) : x_d = \phi(x') \}$.  Finally, the bounds
  \eqref{eq:assumption-a} on the coefficient of \(a\) equally hold for
  \(\tilde a\), and the following estimate holds for \(\tilde f\):
  \begin{equation*}
    \| \tilde f \|_{\fsL^p((- \beta,0); \fsL^{q}( \Omega_1))}
    \le R^{-(\frac{2}p +\frac{d}{q})} \| f \|_{\fsL^p([0,T]; \fsL^q(\Omega))} .
  \end{equation*}

  Hence, using the notation \(V := B_d(0,1) \cap \{(x',x_d) : x_d > \phi(x')\}\), we
  get thanks to \cref{thm:oscillation-decay-boundary} applied to the function
  $ \frac{\tilde{w}}{\oscillation_{(-\beta,0] \times U} \tilde w} \in [0,1]$, that (when
  $t_0 \in [\beta\,R^2, T]$ and $R \in (0,R_0)$)
  \begin{equation}\label{eq:oscillation-decay-w}
    \begin{aligned}
      \oscillation_{(t_0-\beta R^2/16,t_0] \times [B_d(x_0,R/4) \cap \Omega]} w
      &=
        \oscillation_{(-\beta/16,0] \times [B_d(0,1/4) \cap V]} \tilde w\\
      &=
        \oscillation_{(-\beta/16,0] \times [B_d(0,1/4) \cap V]} \left(\frac{\tilde{w}}{\oscillation_{(-\beta,0] \times V} \tilde w} \right)\,\,   \oscillation_{(-\beta,0] \times V} \tilde w \\
      &\le  \bigg[ 1-\delta + \frac{C_{f,1} \, R^2 \, \| \tilde f \|_{\fsL^p((- \beta,0); \fsL^{q}( \Omega_1))} }{\oscillation_{(-\beta,0] \times V} \tilde w} \, \bigg]
        \,  \oscillation_{(-\beta,0] \times V} \tilde w  \\
      &\le (1-\delta)
        \oscillation_{(t_0-\beta R^2,t_0] \times [B_d(x_0,R) \cap \Omega]} w
        + C_{f,1} \, R^{\gamma} \, \| f \|_{\fsL^p([0,T]; \fsL^q(\Omega))},
    \end{aligned}
  \end{equation}
  where we recall that \(\gamma = 2-\frac{2}{p} - \frac{d}{q}\) from the statement.

  We then iterate this estimate in order to find the claimed Hölder regularity.
  As a starting point, note that we control \(\| w \|_{\infty}\) by the first part of
  \cref{thm:supremum-bound}.  For the induction, set
  \begin{equation}
    \label{eq:def-rf}
    R_f = R_0 \left( 1 + \frac{\| f \|_{\fsL^p([0,T]; \fsL^q(\Omega))}}
      {\| w \|_{\infty}+\| w_{\init} \|_{\mathrm{Lip}}} \right)^{-\frac{1}{\gamma}}
    \le R_0 ,
  \end{equation}
  and fix \(\epsilon \in (0,1/2)\) for the initial regularity in the second part of
  \cref{thm:supremum-bound}.  Then set
  \begin{equation}
    \label{eq:def-lambda-0}
    \Lambda = \max \left( 1 - \frac{\delta}{2}, 4^{-\gamma}, 4^{-1+2\epsilon} \right),
  \end{equation}
  and
  \begin{equation}
    \label{eq:def-c1}
    C_1 = \max \left(1,
      K_2\, (4R_0)^{1-2\epsilon},
      K_3\, (4R_0)^{\gamma},
      \frac{2}{\delta}\, C_{f,1}\, R_{0}^{\gamma} \right).
  \end{equation}
  For any fixed \((t_0,x_0) \in [0,T] \times \overline \Omega\), we now claim that for all
  integers \(k\ge0\) such that \(\beta R_{f}^2 4^{-2k} \le t_0\), it holds that
  \begin{equation}\label{eq:oscillation-induction}
    \oscillation_{(t_0-\beta R_f^2 4^{-2k},t_0] \times B_d(x_0,R_f 4^{-k}) \cap \Omega} w
    \le C_1 \,   (\| w \|_{\infty} + \| w_{\init} \|_{\mathrm{Lip}})\, \Lambda^k,
  \end{equation}

  For the proof of the claim~\eqref{eq:oscillation-induction}, we define
  $k_0 := \min \{ k \ge 0 \text{ such that } \beta R^2_f4^{-2k} \le t_0 \}$.

  In the case when \(k_0=0\), we use that \(0 \le w \le \| w \|_{\infty}\), which yields
  \eqref{eq:oscillation-induction}, since the oscillation is at most
  \(\|w\|_{\infty}\).

  In the case when \(k_0 \ge 1\), recording the definition of $k_0$,  we see that
  \(t_0 \le \beta R_f^2 4^{-2k_0+2}\).  Hence, using the second part of
  \cref{thm:supremum-bound} with $R := \sqrt{t_0/\beta} \le R_f 4^{-k_0+1}$, we find
  \begin{equation*}
    \begin{split}
      &\oscillation_{(t_0-\beta R_f^2 4^{-2k_0},t_0] \times [B_d(x_0,R_f 4^{-k_0}) \cap \Omega]} w\\
      &\le
        \oscillation_{(0, \beta\,R^2] \times [B_d(x_0,R) \cap \Omega]} w\\
      &\le K_2 \, R^{1-2\epsilon}
        \| w_{\init} \|_{\mathrm{Lip}}
        + K_3 \, R^{\gamma}\,
        \| f \|_{\fsL^p([0,T]; \fsL^q(\Omega))} \\
      &\le K_2 (4R_0)^{1-2\epsilon}
        \left( \frac{1}{4^{1-2\epsilon}} \right)^{k_0}\,
        \| w_{\init} \|_{\mathrm{Lip}}
        + K_3
        \left( 4R_0 \right)^{\gamma}
        4^{-k_0 \gamma}
        \frac{\| f \|_{\fsL^p([0,T]; \fsL^q(\Omega))}}
        {\left( 1 + \frac{\| f \|_{\fsL^p([0,T]; \fsL^q(\Omega))}}{\| w \|_{\infty}+\| w_{\init} \|_{\mathrm{Lip}}} \right)}
      \\
      &\le  C_1\,  (\| w \|_{\infty}+\| w_{\init} \|_{\mathrm{Lip}})\, \Lambda^{k_0}.
    \end{split}
  \end{equation*}

  For the remaining \(k \ge k_0\), we prove the
  claim~\eqref{eq:oscillation-induction} by induction.  Assume that the claim
  holds for \(k \ge k_0\), then for \(k+1\), we use \eqref{eq:oscillation-decay-w}
  with $R = R_f\,4^{-k} \le R_0$ to find
  \begin{equation*}
    \begin{split}
      &\oscillation_{(t_0-\beta R_f^2 4^{-2k-2},t_0] \times [B_d(x_0,R_f 4^{-k-1}) \cap \Omega]} w\\
      &\le (1 - \delta) \,  \oscillation_{(t_0-\beta R_f^2 4^{-2k},t_0] \times [B_d(x_0,R_f 4^{-k}) \cap \Omega]} w
        +\, C_{f,1} \, (R_f/4^k)^{\gamma} \, \| f \|_{\fsL^p([0,T]; \fsL^q(\Omega))} \\
      &\le (1 - \delta) \, C_1 \,   (\| w \|_{\infty}+\| w_{\init} \|_{\mathrm{Lip}})\,\Lambda^k  +  C_{f,1} \, (R_0/4^k)^{\gamma} \,
        \frac{\| f \|_{\fsL^p([0,T]; \fsL^q(\Omega))}}
        {\left( 1 + \frac{\| f \|_{\fsL^p([0,T]; \fsL^q(\Omega))}}{\| w \|_{\infty}+\| w_{\init} \|_{\mathrm{Lip}}} \right)}.
    \end{split}
  \end{equation*}
  Hence, we see that we get estimate \eqref{eq:oscillation-induction} for \(k+1\)
  instead of $k$, as soon as
  \begin{equation} \label{fif}
    \Lambda \ge 1-\delta  + \frac{C_{f,1} R_0^{\gamma}}{C_1}\,
    \left( \frac{\Lambda^{-1}}{4^{\gamma}} \right)^k.
  \end{equation}
  By the choice of \(C_1\), we find that \(\frac{C_{f,1} R_0^{\gamma} }{C_1}\le \delta/2\)
  and, by the choice of \(\Lambda\), we find that \(\frac{\Lambda^{-1}}{4^{\gamma}} \le 1\), so
  that \eqref{fif} holds as soon as $\Lambda \ge 1 - \delta/2$, which holds, once again
  thanks to the choice of $\Lambda$.  \medskip

  For the resulting Hölder continuity, we consider \((t,x)\) and \((t',x')\) in
  \([0,T] \times \Omega\) such that \(|t-t'| \le r^2\) and \(|x-x'| \le r \le \diam (\Omega) \) holds
  for some \(r>0\).  Without loss of generality, we take \(t' \le t\), so that
  defining $\tilde r := r/\sqrt{\min(1,\beta)}$, the following inequality holds:
  \begin{equation*}
    |w(t,x) - w(t',x')|
    \le
    \oscillation_{[t-\beta \tilde r^2,t]\times [B_d(x,\tilde r) \cap \Omega]} w, \qquad
    \text{when} \quad t \ge \beta \tilde r^2.
  \end{equation*}
  For \(\tilde r \le R_f\), we consider the maximal integer \(k\) such that
  \(R_f4^{-k} \ge \tilde r\).  In the case when \(t \ge \beta R_f^2 4^{-2k}\), we use
  estimate \eqref{eq:oscillation-induction}, and get
  \begin{equation*}
    |w(t,x) - w(t',x')|
    \le
    C_1 \, (\| w \|_{\infty}+\| w_{\init} \|_{\mathrm{Lip}}) \,\Lambda^{k}.
  \end{equation*}
  By the choice of \(k\), we then have \(\tilde r \ge R_f 4^{-k-1}\) so that
  \(k+1 \ge - \log \left( \tilde r/R_f \right) / \log 4\).  Hence, for
  \(\alpha : = \log(\Lambda^{-1})/\log 4 = \min \left( \frac{ \log( \frac1{1- \delta/2}) } {\log 4}, \gamma, 1 - 2 \epsilon \right) \),
  we find in this case
  \begin{equation*}
    \begin{split}
      |w(t,x) - w(t',x')|
      &\le \frac{C_1}{\Lambda} (\| w \|_{\infty}+\| w_{\init} \|_{\mathrm{Lip}})\,\left( \frac{\tilde r}{R_f} \right)^{\alpha}\\
      &\le \frac{C_1}{\Lambda}\, R_0^{- \alpha}\, (\| w \|_{\infty}+\| w_{\init} \|_{\mathrm{Lip}} + \| f \|_{\fsL^p([0,T]; \fsL^q(\Omega))}  )^{\frac{\alpha}{\gamma}}\,  (\| w \|_{\infty}+\| w_{\init} \|_{\mathrm{Lip}} )^{1-\frac{\alpha}{\gamma}} .
    \end{split}
  \end{equation*}
  Recalling the bound of \cref{thm:supremum-bound} for \(\| w \|_{\infty}\), that is
  $$ \| w \|_{\infty} \le K_1\, C_{\Omega, \frac{T}{a_0},d}\, T^{\frac\gamma{2}} \, \| f_+ \|_{\fsL^p([0,T]; \fsL^q(\Omega))} + \| w_{\init} \|_{\infty}, $$
  we end up with the statement of the Theorem.

  In the case when \(t \le \beta R_f^2 4^{-2k}\), we use the second bound of
  \cref{thm:supremum-bound} with \(R := R_f 4^{-k} \ge \tilde r\) to find
  \begin{equation*}
    \begin{split}
      |w(t,x) - w(t',x')|
      &\le
        \oscillation_{[0,t]\times [B_d(x,\tilde r) \cap \Omega]} w \\
      &\le
        \oscillation_{[0,\beta R^2] \times [B_d(x,R) \cap \Omega]} \\
      &\le K_2 \, R^{1-2\epsilon}
        \| w_{\init} \|_{\mathrm{Lip}}
        + K_3 \, R^{\gamma}\,
        \| f \|_{\fsL^p([0,T]; \fsL^q(\Omega))} \\
      &\le K_2 R_f^{1-2\epsilon}
        4^{-k(1-2\epsilon)}
        \| w_{\init} \|_{\mathrm{Lip}}
        + K_3
        R_f^{\gamma}
        4^{-k\gamma}
        (\| w \|_{\infty}+\| w_{\init} \|_{\mathrm{Lip}}) \\
      &\le K_2 4^{1-2\epsilon}\,
        {\tilde r}^{1-2\epsilon}\,
        \| w_{\init} \|_{\mathrm{Lip}}
        + K_3
        4^{\gamma}\,
        {\tilde r}^{\gamma}\,
        (\| w \|_{\infty}+\| w_{\init} \|_{\mathrm{Lip}})  \\
      &\le K_2 4^{1-2\epsilon}\, \diam (\Omega)^{\gamma - \alpha}\, \min(1,\beta)^{- \gamma/2}\,
        {r}^{\alpha}\,
        \| w_{\init} \|_{\mathrm{Lip}}\\
      &\qquad+ K_3
        4^{\gamma}\, \, \diam (\Omega)^{1 - 2\epsilon - \alpha}\, \min(1,\beta)^{- (1 - 2\epsilon)/2}\,
        {r}^{\alpha}\,
        (\| w \|_{\infty}+\| w_{\init} \|_{\mathrm{Lip}}) .
    \end{split}
  \end{equation*}
  In that case also, we end up with the statement of the Theorem.

  Finally, in the case when \(\tilde r > R_{f}\), we can use the supremum bound
  to find that
  \begin{equation*}
    \begin{split}
      |w(t,x) - w(t',x')|
      &\le \| w \|_{\infty}
        \le \| w \|_{\infty} \, R_f^{-\alpha}\, \tilde r^{\alpha}\\
      &\le R_0 ^{- \alpha}\, (\| w \|_{\infty}+\| w_{\init} \|_{\mathrm{Lip}} +
        \| f \|_{\fsL^p([0,T]; \fsL^q(\Omega))}
        )^{\frac{\alpha}{\gamma}}\,  (\| w \|_{\infty}+\| w_{\init} \|_{\mathrm{Lip}} )^{1-\frac{\alpha}{\gamma}} \, \tilde r^{\alpha},
    \end{split}
  \end{equation*}
  which again gives the statement of the Theorem.
\end{proof}

\section{Global strong solutions to the triangular SKT system in dimension 3 and 4} \label{trois}

In this section, we present the proof of Theorem~\ref{proskt}.
\medskip

We first state the most basic a priori estimate for system \eqref{eq:SKT} --
\eqref{skt_id}, considering (sufficiently smooth) solutions $u,v \ge 0$ of the
system.

From the maximum principle applied to the second equation, we get from the start
that (for any $T>0$), \(\| v \|_{\fsL^{\infty}([0,T] \times \Omega)} \le C\).  We denote here and
in the sequel by $C$ a constant depending only on $\Omega$, $T$, the initial data and
the parameters of the system (we will denote by $C_p$ a constant when it also
depends on an extra parameter $p$).

Integrating the first equation on $\Omega$, we also see that
\(\| u \|_{\fsL^{\infty}([0,T] ; \fsL^1(\Omega))} \le C\).  \medskip

 Then, using $u^p$ with $p>0$ as a multiplier in the first equation, we get
$$ \frac{\dd}{\dd t} \int_{\Omega} \frac{u^{p+1}}{p+1} + p \int_{\Omega}  u^{p-1} \nabla u \cdot \left( (d_1 + \sigma\,v) \nabla u + \sigma u\, \nabla v\right) =  \int_{\Omega} u^{p+1} \, ( r_u - d_{11}\, u  - d_{12}\, v ) 
, $$
so that 
$$ \frac{\dd}{\dd t} \int_{\Omega} \frac{u^{p+1}}{p+1} + p  \int_{\Omega} (d_1 + \sigma\,v)  u^{p-1} |\nabla u|^2 
  + \sigma \int_{\Omega} u^p \nabla u \cdot \nabla v  = \int_{\Omega} u^{p+1} \, ( r_u - d_{11}\, u  - d_{12}\, v ) 
, $$
and (using Young's inequality)
$$ \frac{\dd}{\dd t} \int_{\Omega} \frac{u^{p+1}}{p+1}  + d_1\,\frac{4p}{(p+1)^2} \, \int_{\Omega} \,|\nabla (u^{\frac{p+1}{2}})|^2
 \le C_p + C_p \int_{\Omega} u^{p+1} \,|\Delta v|  .$$
Using maximal regularity  in the equation for $v$ (and again Young's inequality), we get
\begin{equation} \label{eqp}
 \int_{\Omega} \frac{u^{p+1}}{p+1} (T) + d_1\,\frac{4p}{(p+1)^2} \, \int_0^T \int_{\Omega} \,|\nabla (u^{\frac{p+1}{2}})|^2 
 \le  C_p + C_p \int_0^T \int_{\Omega} u^{p+2} .
\end{equation}

Proceeding as in \cite{desvillettes-2024-about-shigesada-kawasaki-teramoto}, we
see that if $u$ is bounded in $\fsL^{p + 2} ([0,T] \times \Omega)$, then
$u \in \fsL^{\infty}([0,T]; \fsL^{p+1}(\Omega))$ and
$\nabla (u^{\frac{p+1}{2}}) \in \fsL^2 ([0,T] \times \Omega)$, so that after using a Sobolev
embedding, an interpolation and an induction procedure, we get that
$u \in \fsL^{\infty}([0,T]; \cap_{p \in [1, \infty)} \fsL^p(\Omega))$ as soon as $u$ is bounded in
$\fsL^q([0,T] \times\Omega)$ for some $q> 1 + \frac{d}2$.  \medskip

Thanks to the improved duality lemma of
\cite{canizo-desvillettes-fellner-2014-improved}, we know that (denoting
$\fsL^{r+0} := \cup_{s>0} \fsL^{r+s}$ here and in the rest of the proof)
\begin{equation} \label{ul2}
  u\in \fsL^{2+0}([0,T]\times\Omega),
\end{equation}
(cf.\ \cite{desvillettes-2024-about-shigesada-kawasaki-teramoto}), so that we
get that $u$ is bounded in $\fsL^{\infty}([0,T]; \cap_{p \in [1, \infty)} \fsL^p(\Omega))$ in the
case when $d=1$ or $d=2$.  Standard arguments involving an approximation
procedure or a continuation result (cf.\
\cite{desvillettes-2024-about-shigesada-kawasaki-teramoto}) yield the result of
Theorem~\ref{proskt} in that case.  See also \cite{lou-ni-wu-1998} for a
slightly different approach for the same result.  \medskip

In order to treat the case of dimension $d=3$ (and $d=4$), we now provide new arguments. We introduce the quantity $m:= m(t,x)$ defined as the solution of the heat equation
$$ \partial_t m - \Delta m = u\,(d_{11}u + d_{12}v), $$
together with the (homogeneous) Neumann boundary condition $\nabla m\cdot \vec n(x) = 0$
on $[0,T] \times \partial\Omega$, and the initial datum $m(0,\cdot) = 0$. By the minimum principle,
it is clear that $m\ge 0$.  \medskip

Defining $\nu := \frac{\mu\,u +m}{u+m} $, we observe that
$\min(1,d_1) \le \nu \le \max(1, d_1+ \sigma\,\|v\|_{\infty})$, and $u+m$ satisfies the equation
$$ \partial_t(u+m) - \Delta ( \nu \, (u+m) ) = r_u\,u, $$
together with the (homogeneous) Neumann boundary condition
$\nabla (\mu\,u + m) \cdot \vec n(x) = 0$ on $[0,T] \times \partial\Omega$, and the initial condition
$(u+m)(0,x) \equiv u_{\init}(x)$. As a consequence, the improved duality lemma of
\cite{canizo-desvillettes-fellner-2014-improved} ensures that
\begin{equation} \label{ulm}
u, m\in \fsL^{2+0}([0,T]\times\Omega),
\end{equation}
where (here and in the sequel) we use the shorthand $L^{p+0} := \cup_{q>p} L^q$.
\par
We now consider the quantity $w := \int_0^t (\mu u + m)$, and observe that $w\ge 0$ and
$\partial_t w \ge 0$. Moreover
$\Delta w = \int_0^t \Delta (\mu u + m) = \int_0^t [\partial_t (u + m) - r_u\,u] = u + m - u_{in} - r_u \,\int_0^t u$.

It satisfies therefore the parabolic equation
\begin{equation}
  \label{eq:evolution-w2}
 \nu^{-1}\, \partial_t w - \Delta w =  u_{\init} +    r_u \,\int_0^t u,
\end{equation}
together with the homogeneous Neumann boundary condition, and the initial condition $w(0, \cdot) = 0$.
\medskip

Observing that $\int_0^t u$ lies in $\fsL^{\infty}([0,T]; \fsL^{2+0}(\Omega))$, we see that
we can use \cref{thm:final-hoelder-regularity} with $p =\infty$ and $q= 2+0$
(recalling that $d \le 4$), and deduce from it that
$\|w\|_{\fsC^{0,\alpha}([0,T] \times \overline{\Omega})} \le C$, for some $\alpha, C>0$.

Then, we use the estimate
$$ 0 \le u \le u+m = \Delta w  +  u_{\init} +    r_u \,\int_0^t u \le \Delta w  +  \|u_{\init}\|_{\infty} +    \frac{r_u}{|\Omega|}
 \,\int_0^t \int_{\Omega} u  + r_u\, \int_0^t \bigg[ u -  |\Omega|^{-1}  \int_{\Omega} u \bigg] \le  \Delta \tilde{w}, $$
where
\begin{equation}\label{stara}
  \tilde{w} := w +  \frac{|x|^2}{2d}\, \bigg(\|u_{\init}\|_{\infty} +  \frac{r_u\, T}{|\Omega|} \,\|u\|_{L^{\infty}([0,T] ; L^1(\Omega))} \bigg) + r_u \, \Delta^{-1} \int_0^t \bigg[ u -  |\Omega|^{-1}  \int_{\Omega} u \bigg],
\end{equation}
and  $\Delta^{-1}$ is defined as the operator going from the subset of $\fsL^2(\Omega)$ consisting of functions with $0$-mean value towards itself, which to a function associates the (unique) solution of the Poisson equation with (homogeneous) Neumann boundary condition.
\medskip

Observing that $ \int_0^t \bigg[ u - |\Omega|^{-1} \int_{\Omega} u \bigg]$ lies in
$\fsL^{\infty}([0,T]; L^{2+0}(\Omega))$, we see that
$\Delta^{-1} \int_0^t \bigg[ u - |\Omega|^{-1} \int_{\Omega} u \bigg]$ lies in
$\fsL^{\infty}([0,T]; \fsW^{2,2+\delta}(\Omega))$ for some $\delta>0$ and therefore, in dimension $d \le 4$, thanks
to a Sobolev embedding, in $\fsL^{\infty}([0,T]; \fsC^{0,\alpha}(\overline{\Omega}))$, for some
$\alpha >0$. The same holds for the function $(t,x) \mapsto \frac{|x|2}{2d}$.

We now use  the one-sided interpolation \cref{inpppg}, namely identity
\eqref{inppalk} with $d=3$, $p=2$ and $q= 2\,\frac{3 - \alpha}{2-\alpha}$, that is
\begin{equation}\label{iare}
  \|u \|_{\fsL^{2\,\frac{3 - \alpha}{2-\alpha}} (\Omega)}^3
  \le   C\, \bigg( \,  \| \tilde{w} \|_{\fsC^{\alpha}(\overline{\Omega})}^{\frac3{3-\alpha}}\,
  \| \nabla u\|_{\fsL^2(\Omega)}^{3\frac{2 -\alpha}{3-\alpha}} +
  \| \tilde{w} \|_{\fsC^{\alpha}(\overline{\Omega})}^{3} \,\bigg)   ,
\end{equation}
which holds for any $u, \tilde{w} : \Omega \subset \R^d \to \R$, such that
$0 \le u \le\Delta \tilde{w}$, and $\alpha \in (0,1)$. Here we use it for a given time $t\in [0,T]$.

Recalling estimate \eqref{ul2}, we get that (with $\tilde{w}$ defined by
\eqref{stara})
\begin{equation}\label{arap2}
  \begin{aligned}
    \|u\|_{\fsL^3([0,T] \times \Omega)}^{3}
    &\le  C\, \|u\|^{\alpha}_{\fsL^{2}  ([0,T] \times \Omega)}\, \|u\|_{\fsL^{\frac{3-\alpha}{1 - \alpha/2}} ([0,T] \times \Omega)}^{3 - \alpha} \le  C\,  \|u\|_{\fsL^{\frac{3-\alpha}{1 - \alpha/2}} ([0,T] \times \Omega)}^{3 - \alpha} \\
    &\le C \,\bigg[\, \int_0^T  \|u\|_{\fsL^{\frac{3-\alpha}{1 - \alpha/2}} ( \Omega)}^{\frac{3 - \alpha}{1 - \alpha/2}} \, \bigg]^{1 - \alpha/2} \le C \, \bigg[\, \int_0^T \bigg(\,  \| \tilde{w} \|_{\fsC^\alpha (\overline{ \Omega})}^{\frac{1}{1 - \alpha/2}} \, \, \| \nabla u\|_{\fsL^2(\Omega)}^2 +    \| \tilde{w} \|_{\fsC^\alpha ( \overline{\Omega})}^{\frac{3-\alpha}{1 - \alpha/2}} \bigg) \, \bigg]^{1 - \alpha/2} \\
    &\le C \,  \| \tilde{w} \|_{\fsL^{\infty}([0,T] ; \fsC^\alpha ( \overline{\Omega}))}\, \| \nabla u\|_{\fsL^2([0,T] \times \Omega)}^{2-\alpha}
      +   C \,  \| \tilde{w} \|_{\fsL^{\infty}([0,T] ; \fsC^\alpha ( \overline{\Omega}))}^{3-\alpha}
      \le C   +   C \, \| \nabla u\|_{\fsL^2([0,T] \times \Omega)}^{2-\alpha}.
  \end{aligned}
\end{equation}
Recalling estimate \eqref{eqp} for $p=1$ and using Young's inequality, we see
that $u$ lies in $\fsL^{3}([0,T] \times \Omega)$.

Since (in dimension $d$) we know that
$u \in \fsL^{\infty}([0,T]; \cap_{p \in [1, \infty)} \fsL^p(\Omega))$ as soon as $u$ is bounded in
$\fsL^q([0,T] \times\Omega)$ for some $q> (1 + \frac{d}2)$, we see that
$u \in \fsL^{\infty}([0,T]; \cap_{p \in [1, \infty)} \fsL^p(\Omega))$ when $d=3$. Standard arguments
then show that existence of a strong solution holds in this situation
(cf. \cite{desvillettes-2024-about-shigesada-kawasaki-teramoto}), so that the
proof of \cref{proskt} is complete in dimension $d=3$.  \medskip

For dimension \(d=4\), note that the interpolation in \eqref{arap2} shows that
\(\| u \|_{\fsL^{3+0}([0,T]\times \Omega)} < +\infty\), so that the conclusion holds again.

\section{Global strong solutions to reaction-diffusion systems in dimension 3 and 4}  \label{quatre}

In this section we present results for reaction-diffusion systems which can be
obtained thanks to \cref{thm:final-hoelder-regularity}.

\subsection{A simple proof for the existence of global strong solutions to a quadratic reaction-diffusion system} \label{se41}

We start with a very short  proof of \cref{pron}.
\medskip

\begin{proof}
  We  observe that when $p>0$, (sufficiently smooth) solutions of system \eqref{u14}, \eqref{u14bc} satisfy the identity
  \begin{equation*}
    \frac{\dd}{\dd t} \sum_{i=1}^4 \int_{\Omega} \frac{u_i^{p+1}}{p+1}
    = -p \sum_{i=1}^4 d_i \int_{\Omega} u_i^{p-1}\,|\nabla u_i|^2 + \sum_{i=1}^4 (-1)^i\, \int_{\Omega} u_i^p\, (u_1\,u_3 - u_2\,u_4),
  \end{equation*}
  so that (for some constant $C_p >0$ depending only on $p$, and for all $T>0$), using Young's inequality,
  \begin{equation} \label{u14p}
    \sum_{i=1}^4 \int_{\Omega} \frac{u_i^{p+1}}{p+1} (T)
    + \frac{4p}{(p+1)^2}  \sum_{i=1}^4  d_i \int_0^T \int_{\Omega} |\nabla ( u_i^{\frac{p+1}2} )|^2
    \le  \sum_{i=1}^4 \int_{\Omega} \frac{u_i^{p+1}}{p+1} (0) + C_p \, \sum_{i=1}^4  \int_0^T\int_{\Omega} {u_i^{p+2}} .
  \end{equation}
  We now observe that if for some $p>0$ and $T>0$, one knows that
  $\sum_{i=1}^4 u_i \in \fsL^{p+2}([0,T] \times \Omega)$, then, using estimate \eqref{u14p}, for all
  $i=1,\dots,4$, $u_i \in \fsL^{\infty}([0,T] ; \fsL^{p+1}(\Omega))$ and
  $u_i^{\frac{p+1}2} \in \fsL^2([0,T] ; H^1(\Omega))$, so that, using Sobolev's embedding
  (for $d>2$), $u_i^{\frac{p+1}2} \in \fsL^2([0,T] ; \fsL^{\frac{2d}{d-2}}(\Omega))$ and
  consequently
  $u_i \in \fsL^{p+1}([0,T] ; \fsL^{(p+1)\,\frac{d}{d-2}}(\Omega))$. Interpolating this last
  information with $u_i \in \fsL^{\infty}([0,T] ; \fsL^{p+1}(\Omega))$ and summing over $i$ , we end
  up with $\sum_{i=1}^4 u_i \in \fsL^{(p+1)\, \frac{d + 2}d}([0,T] \times \Omega)$.  Hence we
  improved the integrability if \((p+1) \frac{(d+2)}{d} > p+2\) which is
  equivalent to \(p>\frac{d}{2} - 1\). 

  As a consequence, using an induction, we see that as soon as
  $\sum_{i=1}^4 u_i \in \fsL^{q}([0,T] \times \Omega)$ with $q>1 + \frac{d}2$, then
  $\sum_{i=1}^4 u_i \in \fsL^{\infty}([0,T] ; \cup_{p\in [1, \infty[} \fsL^p( \Omega ))$ (when $d=1$ or $d=2$,
  this is still true though the Sobolev embedding is used sligthly differently in
  this case).

  Then, still assuming that we a priori know that $\sum_{i=1}^4 u_i \in \fsL^{q}([0,T] \times \Omega)$ with
  $q>1 + \frac{d}2$, using maximal regularity, we get an estimate for $\partial_t u_i$
  and $\Delta u_i$ in $ \cup_{p\in [1, \infty[} \fsL^p([0,T] \times \Omega )$. Using a standard  approximation scheme (see for example \cite{canizo-desvillettes-fellner-2014-improved}), we end
  up with the statement of \cref{pron}.
  \medskip

  We have thus reduced the proof of \cref{pron} to proving the a priori estimate: $\sum_{i=1}^4 u_i \in \fsL^{q}([0,T] \times \Omega)$ with
  $q>1 + \frac{d}2$. 
  \medskip

  We now observe that
  $$ \partial_t  \bigg( \sum_{i=1}^4  u_i \bigg) - \Delta \bigg( \sum_{i=1}^4 d_i \, u_i \bigg) = 0, $$
  so that
  \begin{equation} \label{didia}
    \partial_t  \bigg( \sum_{i=1}^4  u_i \bigg) -   \Delta \bigg( \mu\,  \sum_{i=1}^4  u_i \bigg) = 0,
  \end{equation}
  where
  \begin{equation} \label{didi}
    \mu := \frac{ \sum_{i=1}^4 d_i \, u_i }{\sum_{i=1}^4  u_i }
    \in [\min_{i=1,\dots,4} d_i, \max_{i=1,\dots,4} d_i].
  \end{equation}
  Note also that $\sum_{i=1}^4  u_i$ satisfies the homogeneous Neumann boundary condition.

  Using the improved duality lemma of
  \cite{canizo-desvillettes-fellner-2014-improved}, one directly gets that
  $\sum_{i=1}^4 u_i \in \fsL^q([0,T] \times \Omega)$ for some $q>2$, so that when $d=1$ or $d=2$,
  the statement of \cref{pron} is proven.

  \medskip

  We show below how to prove that $\sum_{i=1}^4 u_i \in \fsL^q([0,T] \times \Omega)$ for some
  $q>3$ (in any dimension $d$), so that the statement of \cref{pron} also holds
  when $d=3$ or $d=4$.

  For this, we use estimate \eqref{u14p} when $p=1$, that is
  \begin{equation} \label{u14p1}
    \sum_{i=1}^4 \int_{\Omega} \frac{u_i^{2}}{2} (T) +  \sum_{i=1}^4  d_i \int_0^T \int_{\Omega} |\nabla  u_i|^2 \le  \sum_{i=1}^4 \int_{\Omega} \frac{u_i^{2}}{2} (0) + C_1
    \, \sum_{i=1}^4 \int_0^T\int_{\Omega} {u_i^{3}} ,
  \end{equation}
  and
  introduce $w : = \int_0^t \bigg(\sum_{i=1}^4 d_i \, u_i \bigg)$. We notice that $w\ge 0$ and $\partial_t w \ge 0$. Moreover,
  \begin{equation} \label{newu}
    \Delta w =  \int_0^t  \Delta \bigg(\sum_{i=1}^4 d_i \, u_i \bigg)  = \int_0^t  \partial_t  \bigg( \sum_{i=1}^4  u_i \bigg)  =   \sum_{i=1}^4  u_i  -  \sum_{i=1}^4  u_i^\init ,
  \end{equation}
  so that
  \begin{equation} \label{newuu}
    \partial_t w - \mu\, \Delta w = \sum_{i=1}^4 d_i \, u_i - \frac{ \sum_{i=1}^4 d_i \, u_i }{\sum_{i=1}^4  u_i }\,
    \bigg(  \sum_{i=1}^4  u_i  -  \sum_{i=1}^4  u_i^\init  \bigg)  = \mu\, \sum_{i=1}^4  u_i^\init .
  \end{equation}
  Finally, $\nabla w\cdot \vec n = 0$ on $[0,T] \times \partial\Omega$, and $w(0,\cdot) = 0$.

  As a consequence, we can use \cref{thm:final-hoelder-regularity}, and get that
  (for all $T>0$) $w \in \fsC^{\alpha}([0,T] \times {\overline{\Omega}} )$, for some $\alpha >0$.  Then,
  we use the following interpolation inequality:
  \begin{equation}\label{esalp}
    \| \Delta {w} \|_{\fsL^{2\,\frac{3 - \alpha}{2-\alpha}} (\Omega)}^3 \le  C\, \| {w} \|_{\fsC^{\alpha}(\overline{\Omega})}^{\frac3{3-\alpha}}\, \| \nabla \Delta {w} \|_{\fsL^2(\Omega)}^{3\frac{2 -\alpha}{3-\alpha}} ,
  \end{equation}
  which yields thanks to the identity \eqref{newu}
  \begin{equation} \label{esalp2}
    \bigg\|  \sum_{i=1}^4  u_i  -  \sum_{i=1}^4  u_i^\init \bigg\|_{\fsL^{2\,\frac{3 - \alpha}{2-\alpha}} (\Omega)}^3
    \le  C\, \| {w} \|_{\fsC^{\alpha}(\overline{\Omega})}^{\frac3{3-\alpha}}\, \bigg\| \nabla \bigg[  \sum_{i=1}^4  u_i  -  \sum_{i=1}^4  u_i^\init \bigg]\, \bigg\|_{\fsL^2(\Omega)}^{3\frac{2 -\alpha}{3-\alpha}} .
  \end{equation}
  Then (taking into account the assumption on the initial datum and the estimate
  $w \in \fsC^{\alpha}([0,T] \times {\overline{\Omega}} )$, denoting by $C$ here and in the sequel
  constants depending only on $T$, $\Omega$, $\alpha$, $d_i$ and initial data), we get for
  any $j=1,\dots,4$
  \begin{equation} \label{utl}
    \|u_j\|_{\fsL^{2\,\frac{3 - \alpha}{2-\alpha}} (\Omega)}^3  \le C \bigg(1 +  \bigg\| \nabla \bigg[  \sum_{i=1}^4  u_i \bigg]\, \bigg\|_{\fsL^2(\Omega)}^{3\frac{2 -\alpha}{3-\alpha}} \bigg) .
  \end{equation}
  Noticing first  that $2\,\frac{3 - \alpha}{2-\alpha} > 3$, we see that (for any $j=1,..,4$)
  \begin{equation} \label{utl2}
    \int_{\Omega} u_j^3 \le C\,  \|u_j\|_{\fsL^{2\,\frac{3 - \alpha}{2-\alpha}} (\Omega)}^3 \le  C \, \bigg(1 +  \sum_{i=1}^4 \,\| \nabla    u_i \|_{\fsL^2(\Omega)}^{3\frac{2 -\alpha}{3-\alpha}} \bigg) .
  \end{equation}

  Noticing then that $3\frac{2 -\alpha}{3-\alpha} < 2$, we also see,
  using estimate \eqref{u14p1} and Young's inequality, that
  $\nabla u_i \in \fsL^2([0,T] \times \Omega)$, so that finally, we rewrite estimate \eqref{utl}
  as
  \begin{equation} \label{utfi}
    \|u_j\|_{\fsL^{2\,\frac{3 - \alpha}{2-\alpha}} (\Omega)}^{2\,\frac{3 - \alpha}{2-\alpha}}  \le C \bigg(1 +  \bigg\| \nabla \bigg[  \sum_{i=1}^4  u_i \bigg]\, \bigg\|_{\fsL^2(\Omega)}^{3\frac{2 -\alpha}{3-\alpha}} \bigg)^{\frac23\,\frac{3 - \alpha}{2-\alpha}} \le C  \, \bigg( 1+ \sum_{i=1}^4  \| \nabla    u_i \|^ {2}_{\fsL^2(\Omega)}  \bigg).
  \end{equation}
  Integrating in time, we see that
  \begin{equation} \label{utl3}
    \int_0^T \int_{\Omega} u_j^{2 \, \frac{3 - \alpha}{2-\alpha}} \le C\, \bigg( 1+ \sum_{i=1}^4 \int_0^T  \| \nabla    u_i \|^ {2}_{\fsL^2(\Omega)} \bigg) \le C  \, \bigg( 1+ \sum_{i=1}^4  \| \nabla    u_i \|^ {2}_{\fsL^2([0,T] \times \Omega)} \bigg) \le C  ,
  \end{equation}
  which ends the proof of \cref{pron} in the case when $d=3$ or $d=4$, since
  $2 \frac{3 - \alpha}{2 - \alpha} > 3 = 1 + \frac42$.
\end{proof}

\subsection{General dissipative systems with quadratic intermediate sums}
\label{sec:quadratic-intermediate-sums}

In this subsection, we considerably generalise the result of
Subsection~\ref{se41}, and prove Theorem~\ref{pron2}.  \medskip

We consider reaction-diffusion systems which write
\begin{equation}\label{uifi}
  \partial_t u_i - d_i \Delta u_i = f_i(u_1,\dots,u_m), \qquad i=1,\dots,m,
\end{equation}
together with the Neumann boundary condition and initial condition
\begin{equation}\label{uifi2}
  \nabla u_i \cdot n = 0 \quad {\hbox{ on }} [0,T] \times\partial\Omega, \qquad u_i(0,\cdot) = u^{\init}_i  \qquad i=1,\dots,m ,
\end{equation}
and introduce the general set of conditions:

\medskip
\begin{enumerate}[label={(A\arabic*)}]
  \item\label{ass:mass:local-lipschitz} (Local Lipschitz and quasi-positivity)
        The functions $f_i:\R^m \to \R$ are locally Lipschitz-continuous and satisfy
        the quasi-positivity condition, for $i=1,\dots,m$,
  $$ f_i(u_1,\dots,u_m) \ge 0 \qquad {\hbox{ when }} \qquad u_i =0  \quad {\hbox{ and }}  \quad u_j \ge 0 {\hbox{ for }} j \neq i ; $$
  \item\label{ass:mass:mass-dissipation} (Mass dissipation) There exist
        $\beta_1,\dots,\beta_m>0$ such that for any $u_1,\dots,u_m \ge 0$,
  $$ \sum_{i=1}^m \beta_i \,f_i(u_1,\dots,u_m) \le 0 ; $$
  \item\label{ass:mass:quadratic} (Quadratic intermediate sums) There exists a
        $m \times m$ lower-triangular matrix $A = (a_{ij})_{i,j = 1,\dots,m}$ with
        nonnegative elements $a_{ij}\ge 0$ and strictly positive diagonal
        coefficients $a_{ii} >0$ (so that $A$ is invertible), and $K>0$ such
        that for $i=1,\dots,m$ it holds that
  $$ \sum_{j=1}^i a_{ij}\, f_j(u_1,\dots,u_m) \le K \Big(1 + \sum_{i=1}^m u_i \Big)^2   {\hbox{ for }}  u_i \ge 0. $$
\end{enumerate}

Those conditions are used in the following

\begin{proposition} \label{npp} Let $\Omega$ be a smooth bounded open subset of $\R^d$
  with dimension $d\le 4$.  Further consider coefficients $d_i>0$ for
  $i=1,\dots,m$, and reaction functions $f_i:\R^m \to \R$ satisfying conditions
  \ref{ass:mass:local-lipschitz}, \ref{ass:mass:mass-dissipation} and
  \ref{ass:mass:quadratic}.  We also assume that $u^\init_i \ge 0$ lies in
  $\fsC^2(\overline{\Omega})$ and is compatible with the Neumann boundary condition,
  for $i=1,\dots,m$.

  Then there exists a global (defined on $\R_+ \times \Omega$) and strong (all terms are
  defined a.e.) solution to system \eqref{uifi}, \eqref{uifi2}.
\end{proposition}

We refer to \cite{fellner-morgan-tang-2020-global} for detailed references on
the history of general problems using the same kind of assumptions.  Under
assumptions close to \ref{ass:mass:local-lipschitz},
\ref{ass:mass:mass-dissipation} and the extra assumption
$ |f_j(u_1,..,u_m)| \le K \Big(1 + \sum_{i=1}^m u_i \Big)^{2 + \eta}$ with $\eta>0$ small,
existence of a unique global classical solution to \eqref{uifi} is shown in
\cite{fellner-morgan-tang-2020-global, ,fellner-morgan-tang-2021-uniform},
together with uniformity in large time of the $\fsL^\infty$ bounds.

With assumptions close to \ref{ass:mass:local-lipschitz},
\ref{ass:mass:mass-dissipation}, \ref{ass:mass:quadratic}, the
evolution~\eqref{uifi} is considered in dimension $d=1$ in
\cite{yang-kostianko-sun-tang-zelik-2024-nonconcentration,sun-tang-yang-2023-analysis},
and in that case it is even possible to allow for a cubic (or slightly more than
cubic) intermediate sum condition.  Uniformity in large time of the $\fsL^\infty$
bounds is also provided.  In \cite{morgan-tang-2020-boundedness-lyapunov},
global existence of strong solutions to \eqref{uifi} is obtained in dimension
$d=2$ for assumptions close to \ref{ass:mass:local-lipschitz},
\ref{ass:mass:mass-dissipation}, \ref{ass:mass:quadratic} together with
uniformity in large time of the $\fsL^\infty$ bounds.\medskip

Note that in comparison with the hypothesis
$ f_j(u_1,\dots,u_m) \le K \Big(1 + \sum_{i=1}^m u_i \Big)^2$, hypothesis
\ref{ass:mass:quadratic} is more general and allows to consider a variety of
chemical reaction in dimension $d=3$ and $d=4$.  We give some examples of such
applications.

\begin{remark}
  Similar to the slightly superquadratic growth of the reaction term in
  \cite{fellner-morgan-tang-2020-global}, an inspection of the proof shows that
  we can allow a slightly superquadratic growth in \ref{ass:mass:quadratic},
  i.e.\ we can allow
  $$ \sum_{j=1}^i a_{ij}\, f_j(u_1,\dots,u_m) \le K \Big(1 + \sum_{i=1}^m u_i \Big)^{2+\eta}   {\hbox{ for }}  u_i \ge 0, $$
  for some small \(\eta\) (depending on the different diffusion constants).
\end{remark}

\medskip

\begin{proof}
  Using an idea of \cite{fellner-morgan-tang-2020-global}, we first notice that
  the assumption \ref{ass:mass:mass-dissipation} can be replaced by the stronger
  assumption $ \sum_{i=1}^m \beta_i\, f_i(u_1,\dots,u_m) = 0$.  Indeed one can add an extra
  unknown $u_{m+1}$ satisfying the equation
  \begin{equation}\label{uifi3}
    \partial_t u_{m+1} - d_1 \Delta u_{m+1} = f_{m+1}(u_1,..,u_m), \qquad i=1,\dots,m,
  \end{equation}
  together with the Neumann boundary condition and initial condition
  \begin{equation}\label{uifi4}
    \nabla u_{m+1} \cdot n = 0 \quad {\hbox{ on }} [0,T] \times\partial\Omega, \qquad u_{m+1}(0,\cdot) = 0,
  \end{equation}
  where $f_{m+1} := -\sum_{i=1}^m \beta_i\, f_i \ge 0$.

  One can check that the new system indeed satisfies conditions
  \ref{ass:mass:local-lipschitz}, \ref{ass:mass:mass-dissipation},
  \ref{ass:mass:quadratic}, where moreover
  $ \sum_{i=1}^{m+1} \beta_i f_i(u_1,\dots,u_m) = 0$. In order to check
  condition~\ref{ass:mass:quadratic}, one defines $a_{m+1,i} := \beta_i$ for
  $i=1,\dots,m$ and $a_{m+1,m+1} := 1$.  One also defines $\beta_{m+1}:=1$.
\medskip

  Then, we consider the change of variables $v:= Au$ (with
  $u:= (u_1,\dots,u_m)$, $v := (v_1,\dots,v_m)$), and define
  $B= (b_{ij})_{i,j = 1..m}$ as the inverse of $A$. We observe that thanks to
  the quasi-positivity (assumption~\ref{ass:mass:local-lipschitz}), $u$ is
  nonnegative (and so is $v$ since $v=Au$). Moreover, $v$ satisfies the system
  \begin{equation}\label{uifi5}
    \partial_t v_i - d_i \Delta v_i = g_i(u_1,..,u_m) + \sum_{k<i} c_{ik}\, \Delta v_k, \qquad i=1,\dots,m,
  \end{equation}
  together with the Neumann boundary condition and initial condition
  \begin{equation}\label{uifi6}
    \nabla v_i \cdot n = 0 \quad {\hbox{ on }} [0,T] \times\partial\Omega, \qquad v_i(0,\cdot) = (A \,u^{\init})_i  \qquad i=1,\dots,m ,
  \end{equation}
  where $g = Af$ (with $f:= (f_1,\dots,f_m)$, $g := (g_1,\dots,g_m)$), and
  $c_{ik} := \sum_{j=k}^{i-1} (d_j-d_i)\, a_{ij}\, b_{jk}$ (for $k<i$).  \medskip

  An energy estimate immediately yields (for all $i=1,\dots,m$, and where $\lesssim$
  means ``less than a constant times'', where the constant depends only on the
  data of the problem and $T>0$ such that $t \in [0,T]$)
  $$ \int_{\Omega} v_i(t)^2 + \int_0^t\int_{\Omega}  |\nabla v_i|^2  \lesssim \int_{\Omega} v_i(0)^2 +  \int_0^t\int_{\Omega}g_i\, v_i + \sum_{j < i}  \int_0^t\int_{\Omega} |\nabla v_i \cdot \nabla v_j|, $$
  so that thanks to assumption~\ref{ass:mass:quadratic} and Young's inequality,
  $$ \int_{\Omega} v_i(t)^2 + \int_0^t\int_{\Omega}  |\nabla v_i|^2  \lesssim \int_{\Omega} v_i(0)^2 +  \sum_{j=1}^m \int_0^t\int_{\Omega}( 1 + |v_j|^3) + \sum_{j < i}  \int_0^t\int_{\Omega} | \nabla v_j|^2.$$
  Thanks to an induction, we get for all $i=1,\dots,m$,
  $$
  \int_{\Omega} v_i(t)^2 + \int_0^t\int_{\Omega}  |\nabla v_i|^2  \lesssim \sum_{j=1}^m \int_{\Omega} v_j(0)^2 +  \sum_{j=1}^m \int_0^t\int_{\Omega}( 1 + |v_j|^3) .
  $$
  Using then another induction, we also get for all $i=1,\dots,m$,
  \begin{equation}\label{fff}
    \int_{\Omega} u_i(t)^2 + \int_0^t\int_{\Omega}  |\nabla u_i|^2
    \lesssim   \sum_{j=1}^m  \left(\int_{\Omega} u_j(0)^2 +  \int_0^t\int_{\Omega}( 1 + |u_j|^3) \right) .
  \end{equation}
  At this point, the proof becomes almost identical to the proof proposed in the
  previous subsection.  We introduce
  $w : = \int_0^t \bigg(\sum_{i=1}^m d_i \,\beta_i\, u_i \bigg)$ and
  $\mu := \frac{ \sum_{i=1}^m d_i \, \beta_i\, u_i }{\sum_{i=1}^m \beta_i\, u_i }$. We notice
  that $w\ge 0$ and $\partial_t w \ge 0$. Moreover,
  \begin{equation} \label{newupr}
    \Delta w =  \int_0^t  \Delta \bigg(\sum_{i=1}^m d_i \, \beta_i\, u_i \bigg)  = \int_0^t \bigg[ \partial_t  \bigg( \sum_{i=1}^m \beta_i\,  u_i \bigg)  +  \sum_{i=1}^m \beta_i\, f_i  \bigg] =   \sum_{i=1}^m  \beta_i\, u_i  -  \sum_{i=1}^m \beta_i\, u_i^\init ,
  \end{equation}
  so that
  \begin{equation} \label{newuupr}
    \partial_t w - \mu\, \Delta w = \sum_{i=1}^m d_i \, \beta_i \,u_i - \frac{ \sum_{i=1}^m d_i \, \beta_i\, u_i }{\sum_{i=1}^m  \beta_i\, u_i }\,
    \bigg(  \sum_{i=1}^m  \beta_i\, u_i  -  \sum_{i=1}^m \beta_i\, u_i^\init  \bigg)  = \mu\, \sum_{i=1}^m \beta_i\, u_i^\init .
  \end{equation}
  Finally, $\nabla w\cdot \vec n = 0$ on $[0,T] \times \partial\Omega$, and $w(0,\cdot) = 0$.

  As a consequence, we can use \cref{thm:final-hoelder-regularity}, and get for
  all $T>0$ that $w \in \fsC^{\alpha}([0,T] \times {\overline{\Omega}} )$, for some $\alpha >0$.
  Then, we use interpolation inequality \eqref{esalp} which yields thanks to the
  identity \eqref{newupr} (for any $t \in [0,T]$)
  \begin{equation} \label{esalp2pr}
    \bigg\|  \sum_{i=1}^m \beta_i\,  u_i  -  \sum_{i=1}^m \beta_i\,  u_i^\init \bigg\|_{\fsL^{2\,\frac{3 - \alpha}{2-\alpha}} (\Omega)}^3
    \le  C\, \| {w} \|_{\fsC^{\alpha}(\overline{\Omega})}^{\frac3{3-\alpha}}\, \bigg\| \nabla \bigg[  \sum_{i=1}^m \beta_i\,  u_i  -  \sum_{i=1}^m \beta_i\,  u_i^\init \bigg]\, \bigg\|_{\fsL^2(\Omega)}^{3\frac{2 -\alpha}{3-\alpha}} ,
  \end{equation}
  and, for any $j=1,\dots,m$ (and any $t \in [0,T]$),
  \begin{equation} \label{utlpr}
    \int_{\Omega} u_j^3 \le C\,  \|u_j\|_{\fsL^{2\,\frac{3 - \alpha}{2-\alpha}} (\Omega)}^3  \le C \bigg(1 +  \bigg\| \nabla \bigg[  \sum_{i=1}^m  u_i \bigg]\, \bigg\|_{\fsL^2(\Omega)}^{3\frac{2 -\alpha}{3-\alpha}} \bigg)  \le  C \, \bigg(1 +  \sum_{i=1}^m \,\| \nabla    u_i \|_{\fsL^2(\Omega)}^{3\frac{2 -\alpha}{3-\alpha}} \bigg) .
  \end{equation}
  Noticing then that $3\frac{2 -\alpha}{3-\alpha} < 2$, we also see, using estimate
  \eqref{fff} and Young's inequality, that $\nabla u_i \in \fsL^2([0,T] \times \Omega)$, so that
  finally, we can use estimate \eqref{utlpr} to get
  \begin{equation}
    \|u_j\|_{\fsL^{2\,\frac{3 - \alpha}{2-\alpha}} (\Omega)}^{2\,\frac{3 - \alpha}{2-\alpha}}  \le C \bigg(1 +  \bigg\| \nabla \bigg[  \sum_{i=1}^m  u_i \bigg]\, \bigg\|_{\fsL^2(\Omega)}^{3\frac{2 -\alpha}{3-\alpha}} \bigg)^{\frac23\,\frac{3 - \alpha}{2-\alpha}} \le C  \, \bigg( 1+ \sum_{i=1}^m  \| \nabla    u_i \|^ {2}_{L^2(\Omega)}  \bigg).
  \end{equation}
  Integrating in time, we see that
  \begin{equation}
    \int_0^T \int_{\Omega} u_j^{2 \, \frac{3 - \alpha}{2-\alpha}} \le C\, \bigg( 1+ \sum_{i=1}^m \int_0^T  \| \nabla    u_i \|^ {2}_{L^2(\Omega)} \bigg) \le C  \, \bigg( 1+ \sum_{i=1}^m \| \nabla    u_i \|^ {2}_{L^2([0,T] \times \Omega)} \bigg) \le C  .
  \end{equation}
  We have thus shown for all $j=1,\dots,m$ the estimates
  $\|u_j\|_{\fsL^{3+0}([0,T] \times \Omega) }\le C$, $\|u_j\|_{\fsL^{\infty}([0,T] ; \fsL^2(\Omega))} \le C$,
  $\|u_j\|_{\fsL^{2}([0,T] ; \fsH^1(\Omega))} \le C$.  \medskip

  We start then an induction, defining $p_1 := 3 + \delta$ for some small $\delta>0$.  We
  already know that $\|u_j\|_{\fsL^{p_1 + 0}([0,T] \times \Omega) } \le C$ for all
  $j=1,\dots,m$.

  Suppose that for some $p_n>1$, one has for all $j=1,\dots,m$ the estimate
  $\|u_j\|_{\fsL^{p_n + 0}([0,T] \times \Omega) } \le C$.  This implies
  $\|v_j\|_{\fsL^{p_n +0}([0,T] \times \Omega) } \le C$ for all \(j=1,\dots,m\).  Using
  equation \eqref{uifi5} for $i=1$ and assumption~\ref{ass:mass:quadratic}, we
  see that $\partial_t v_1 - d_1\, \Delta v_1$ is bounded in
  $\fsL^{p_n/2 + 0}([0,T] \times \Omega)$. Using the properties of the heat equation, we
  see that $\|v_1\|_{\fsL^{p_{n+1} + 0}([0,T] \times \Omega) } \le C$, with
  $\frac1{p_{n+1}} = \frac2{p_n} - \frac2{d+2}$ in dimension $d$ (provided that
  $p_{n+1} \in [1, \infty)$).  The maximal regularity estimate for the heat equation
  also implies that $\Delta v_1$ is bounded in
  $\fsL^{p_n/2 + 0}([0,T] \times \Omega)$. Equation \eqref{uifi5} for $i=2$ and
  assumption~\ref{ass:mass:quadratic} ensures that $\partial_t v_2 - d_2\, \Delta v_2$ is
  bounded in $\fsL^{p_n/2 + 0}([0,T] \times \Omega)$, so that
  $\|v_2\|_{\fsL^{p_{n+1} + 0}([0,T] \times \Omega) } \le C$ and $\Delta v_2$ is bounded in
  $\fsL^{p_n/2 + 0}([0,T] \times \Omega)$. Proceeding in the same way for $v_3,\dots,v_m$,
  we see for all $j=1,\dots,m$ that
  $\|v_j\|_{\fsL^{p_{n+1} + 0}([0,T] \times \Omega) } \le C$.  Finally, for all $j=1,\dots,m$
  it holds that $\|u_j\|_{\fsL^{p_{n+1} + 0}([0,T] \times \Omega) } \le C$.

  Observing that the sequence $(1/p_n)$ is decreasing (as long at it is defined)
  when $p_1 > 1 + \frac{d}{2}$, which is true when the dimension is $d=3$ or
  $d=4$, we obtain for all $j=1,\dots,m$ that $\|v_j\|_{\fsL^{p}([0,T] \times \Omega) } \le C$
  and $\|\Delta v_j\|_{\fsL^{p}([0,T] \times \Omega) } \le C$ for all $p\in [1, \infty)$, and in fact
  (using the properties of the heat equation one last time)
  $\|v_j\|_{\fsL^\infty([0,T] \times \overline{\Omega} ) } \le C$. Finally, we obtain
  $\|u_j\|_{\fsL^\infty([0,T] \times \overline{\Omega} ) } \le C$.  \medskip

  Standard approximation or continuation procedures (using assumption
  \ref{ass:mass:local-lipschitz}) allow then to conclude the proof of
  Proposition~\ref{npp}.
\end{proof}

Among systems satisfying the assumptions of Proposition~\ref{npp}, we exhibit the following ones.
\medskip

First, we recall system \eqref{uum} introduced in the introduction, coming out
from the reversible chemical reaction
$b_1S_1 + \dots + b_mS_m \rightleftharpoons S_{m+1} + S_{m+2}$ where, $m \in\N\setminus\{0\}$ and for
$i \in 1,\dots,m$, we have $b_i \in\N\setminus\{0\}$. This system writes
\begin{equation}
  \left\{ \begin{lgathered}
  \partial_t u_{m+1} - d_{m+1} \Delta u_{m+1} = \prod_{i=1}^m u_i^{b_i} - u_{m+1}\,u_{m+2}, \\
  \partial_t u_{m+2} - d_{m+2} \Delta u_{m+2} =  \prod_{i=1}^m u_i^{b_i} - u_{m+1}\,u_{m+2},  \\
  \partial_t u_i - d_i \Delta u_i
  = -  b_i\,( \prod_{i=1}^mu_i^{b_i} - u_{m+1}\,u_{m+2})
  \quad\text{ for }\quad i=1,\dots,m.
 \end{lgathered} \right.
\end{equation}

The proof of Theorem~\ref{pron2} is then very short: We see indeed that
assumption \ref{ass:mass:mass-dissipation} holds with $\beta_{m+1}=\beta_{m+2} = 1$ and
$\beta_i = \frac{2}{mb_i}$ for $i\leq m$, while assumption \ref{ass:mass:quadratic}
holds with $A = (a_{ij})_{1\leq i, j \leq m}$ where $a_{ii} = 1$ for $1 \leq i \leq m + 2$,
$a_{m+1, 1} = a_{m+2, 1}= \frac{1}{b_1}$ and $a_{ij} = 0$ otherwise. As a
consequence, Proposition~\ref{npp} can be used, which concludes the proof of
Theorem~\ref{pron2}.  \medskip

We can also extract some systems from
\cite{morgan-tang-2020-boundedness-lyapunov}, which have a structure close to
the previous system, and extend existence of global strong solutions to
dimension $d=3$ and $d=4$ for those systems.

We first write down such a system, coming out from the reversible chemical
reaction $S_1 + 2S_2 \rightleftharpoons S_2 + S_3$, which writes
\begin{equation}
  \left\{ \begin{array}{l}
\partial_t u_1 - d_1 \Delta u_1 = u_2\, u_3 - u_1\,u_2^2, \\
  \partial_t u_2 - d_2 \Delta u_2 =  u_2\, u_3 - u_1\,u_2^2,  \\
 \partial_t u_3 - d_3 \Delta u_3 = -  \,( u_2\, u_3 - u_1\,u_2^2).
 \end{array} \right.
\end{equation}
We see that assumption \ref{ass:mass:mass-dissipation} holds with $\beta_1=\beta_2 = 1$
and $\beta_3 = 2$, while assumption \ref{ass:mass:quadratic} holds with
\begin{equation}
  A :=  \left( \begin{array}{ccc}
 1 & 0 & 0 \\
 0 & 1 & 0 \\
 1 & 1 & 2
 \end{array} \right). 
\end{equation}

We then write down a system coming out from the reversible chemical reaction
$p\, S_1 + q \,S_2 \rightleftharpoons 2\, S_3$ (with $p, q\in \N \setminus \{0\}$), which writes
\begin{equation}
  \left\{ \begin{lgathered}
      \partial_t u_1 - d_1 \Delta u_1 = p \,(u_3^2 - u_1^p\,u_2^q), \\
      \partial_t u_2 - d_2 \Delta u_2 = q \,(u_3^2 - u_1^p\,u_2^q),  \\
      \partial_t u_3 - d_3 \Delta u_3 = - 2 \,(u_3^2 - u_1^p\,u_2^q).
    \end{lgathered} \right.
\end{equation}
We see that assumption \ref{ass:mass:mass-dissipation} holds with $\beta_1=\beta_2 = 2$
and $\beta_3 = p+q$, while assumption \ref{ass:mass:quadratic} holds with
\begin{equation}
  A :=  \left( \begin{array}{ccc}
 1 & 0 & 0 \\
 0 & 1 & 0 \\
 2 & 2 & p+q
 \end{array} \right). 
\end{equation}

For the systems described above, Proposition~\ref{npp} enables (as in the proof
of Theorem~\ref{pron2}) to get global strong solutions in dimension $d=3$ and
$d=4$ for all diffusion rates $d_i >0$, whereas such solutions were built only
in dimension $d=1$ or $d=2$, or in higher dimension for diffusion rates
satisfying constraints, in \cite{morgan-tang-2020-boundedness-lyapunov}. Note
however that the conditions on the initial data considered in
Proposition~\ref{npp} are significantly more stringent than those of
\cite{morgan-tang-2020-boundedness-lyapunov}.  Finally, we indicate that
existence of global strong solutions for superquadratic systems in high
dimension is not always possible, since very interesting examples of blowups
exist, cf.\ \cite{pierre-schmitt-2000-blow,pierre-schmitt-2023-examples}.

\appendix

\section{Proof of the one-sided interpolation estimates}\label{AB}

In this section, we present a proof of Proposition~\ref{inpppg}, which consists
of the interpolation inequalities \eqref{inppalk1} and \eqref{inppalk}.  Those
interpolations are used in order to obtain inequality \eqref{iare}.  \medskip

We first capture the geometry by the Lemma below.
\begin{lemma}\label{thm:star-shaped-radius}
  Let $\Omega \subset \R^d$ be a bounded \(\fsC^2\) domain.  Then there exists \(R_0>0\)
  such that for any \(0<R\le R_0\) and any \(x_0 \in \Omega\), there exists
  \(\hat x_0 \in \Omega\) such that \(B_d(\hat x_0,R/4) \subset \Omega\) and \(B_d(x_0,R) \cap \Omega\) is
  star-shaped with core \(B_d(\hat x_0,R/4)\).
\end{lemma}

\begin{proof}
  Like in the proof of \cref{thm:final-hoelder-regularity}, we use that the  \(\fsC^2\) 
  regularity of the domain  implies that for a small enough \(R_0\), the boundary can be
  described by a graph corresponding to a function with arbitrary small gradient.  In this setting, the
  claimed point \(\hat x_0\) can always be found.
\end{proof}

The next building block of the proof is the following Lemma.
\begin{lemma}\label{thm:estimate-ball}
  Let $\Omega \subset \R^d$ be a bounded \(\fsC^2\) domain and \(R_0>0\) obtained from
  \cref{thm:star-shaped-radius}.  Then for any \(x_0 \in \Omega\) and \(0<R\le R_0\), in
  the cut ball \(B := B_d(x_0,R) \cap \Omega\), any functions $u: B \to \R^+$, $w: B \to \R$
  satisfying $u \le \Delta w$ in \(B\) can for any \(p,q \in [1,\infty)\) be estimated as
  \begin{equation}
    \label{eq:interpolation:ball-estimate}
    \| u \|_{\fsL^q(B)}
    \le C_{d,p,q}\, \bigg[
    R^{1-d \left( \frac{1}{p} - \frac{1}{q} \right)}
    \| \nabla u \|_{\fsL^p(B)}
    +
    R^{-2+\frac{d}{q}}
    \| w \|_{\fsL^{\infty}(B)} \,\bigg] ,
  \end{equation}
  where $C_{d,p,q}>0$ is a constant depending only on $d,p,q$.

  In the same way, for any $\alpha \in (0,1)$,
  \begin{equation}
    \label{eq:interpolation:ball-estimate-alpha}
    \| u \|_{\fsL^q(B)}
    \le C_{d,p,q, \alpha}\, \bigg[
    R^{1-d \left( \frac{1}{p} - \frac{1}{q} \right)}
    \| \nabla u \|_{\fsL^p(B)}
    +
    R^{-2+ \alpha + \frac{d}{q}}
    \| w \|_{\fsC^{\alpha}( \overline{B} )} \,\bigg] ,
  \end{equation}
  where $C_{d,p,q, \alpha}>0$ is a constant depending only on $d,p,q, \alpha$.
\end{lemma}
\begin{proof}
  We consider a smooth weight function \(\chi :\R^d \to \R^+\) such that
  \(\supp \chi \subset B_d(0,1/4)\) and \(\int \chi = 1\), and its scaled version
  \begin{equation*}
    \chi_R(x) = R^{-d} \chi(x/R),
  \end{equation*}
  so that \(\supp \chi_R \subset B_d(0,R/4)\) and \(\int \chi_R = 1\).

  Given \(x_0 \in \Omega\) and \(0 < R \le R_0\), we consider some \(\hat x_0\) given by
  \cref{thm:star-shaped-radius} and define the average
  \begin{equation*}
    \bar u := \int_{x \in B_d(0,R/4)} u(\hat x_0-x)\, \chi_R(x)\, \dd x.
  \end{equation*}
  We then estimate the \(\fsL^q\) norm of $u$ as
  \begin{equation*}
    \| u \|_{\fsL^q(B)}
    \le \| u - \bar u \|_{\fsL^q(B)}
    + \| \bar u \|_{\fsL^q(B)}.
  \end{equation*}

  For the second term, we estimate (with $c_d$ the volume of $B_d(0,1)$)
  \begin{equation}\label{presti}
    \begin{aligned}
      \| \bar u \|_{\fsL^q(B)}^q
      &= |B|\, |\bar u|^q
        \le c_d\,R^d \, \bigg( \int \chi_R(x)\, u(\hat x_0 - x)\, \dd x\,\bigg)^q\\
      &\le c_d\,R^d \, \bigg( \int \chi_R(x)\, \Delta w(\hat x_0-x)\, \dd x\,\bigg)^q
        \le c_d\,R^d \, \bigg( \int \Delta \chi_R(x)\,  w(\hat x_0-x)\, \dd x\,\bigg)^q\\
      &\le  c_d\,R^d \, \bigg( R^{-d-2} \int \Delta \chi ( \frac{x}R )\,  w(\hat x_0-x)\, \dd x\,\bigg)^q
        \le c_d\, R^{d-2q} \, \|\Delta \chi\|_{\fsL^1(\R^d)}^q\, \|w\|_{\fsL^{\infty}(B)}^q.
    \end{aligned}
  \end{equation}
  For the Hölder regularity, the estimate can be changed as
  \begin{equation}
    \begin{aligned}
      \| \bar u \|_{\fsL^q(B)}^q
      &\le c_d\,R^d \, \bigg( \int_{B_d(0,R/4)} \Delta \chi_R(x)\,  w(\hat x_0-x)\, \dd x\,\bigg)^q\\
      &\le c_d\,R^d \, \bigg( \int_{B_d(0,R/4)} \Delta \chi_R(x)\,  \bigg[w(\hat x_0-x)
        - |B_d(0,R/4)|^{-1} \int_{B_d(0,R/4)} w(\hat x_0-y)\,\dd y \, \bigg] \, \dd x\,\bigg)^q\\
      &= 2^{-q}\, c_d\,R^{d}\, \bigg( |B_d(0,R/4)|^{-1}\, \int_{B_d(0,R/4)}\int_{B_d(0,R/4)} |x-y|^{\alpha} \,
        [\Delta \chi_R(x) - \Delta \chi_R(y)] \,
        \frac{w(\hat x_0-x) - w(\hat x_0-y)}{|x - y|^{\alpha} } \, \dd x\, \dd y  \,\bigg)^q\\
      &\le 2^{q} (c_d\, R^{d})^{1-q} \, \|w\|_{\fsC^{\alpha}(\overline{B})}^q\,
        \bigg(\, R^{-2 - d} \int_{B_d(0,R/4)}\int_{B_d(0,R/4)} |x-y|^{\alpha} \,
        \bigg|\Delta \chi ( \frac{x}R) - \Delta \chi (\frac{y}R) \,\bigg|  \, \dd x\, \dd y \,\bigg)^q\\
      &=  2^{q}\, c_d^{1 -q}\, R^{d - (2 - \alpha)q} \, \|w\|_{\fsC^{\alpha}(\overline{B})}^q\,
        \bigg(\,  \int\int |\xi-\eta|^{\alpha} \,
        \bigg|\Delta \chi ( \xi) - \Delta \chi (\eta) \,\bigg|  \, \dd \xi\, \dd \eta \,\bigg)^q .
    \end{aligned}
  \end{equation}
   Note that here, the quantities $\|\Delta \chi\|_{\fsL^1(\R^d)}$ and $ \int\int |\xi-\eta|^{\alpha} \,
        \bigg|\Delta \chi ( \xi) - \Delta \chi (\eta) \,\bigg|  \, \dd \xi\, \dd \eta$ 
  are just  dimensional constants.
  \medskip

  For the first term, we use a Poincaré-Sobolev inequality, that we prove thanks
  to elementary computations. We first observe that for $x\in B$,
  \begin{equation*}
    u(x) - \bar u
    = \int [u(x) - u(y)]\, \chi_R(\hat x_0-y)\, \dd y
    = \int \int_{s=0}^1 \nabla u (x + s(y-x)) \cdot (x-y)\,  \chi_R(\hat x_0-y)\,\dd s\, \dd y,
  \end{equation*}
  where we use the fundamental theorem of calculus and the fact that \(B\) is star-shaped
  with the core \(B_d(\hat x_0,R/4)\).  Hence, using the change of variable
  $z := s \,(y-x)$,
  \begin{equation*}
    \begin{aligned}
      |u(x) - \bar u|
      &\le \int_{x + z \in B} \int_{s=0}^1  |\nabla u (x + z)| \frac{|z|}{s} \,
        \left|\chi_R \left(\hat x_0 -x - \frac{z}s \right)\right|\, \frac{\dd s}{s^d}\, \dd z\\
      &\le  R^{-d} \|\chi\|_{\infty}
        \int_{x + z \in B} |\nabla u (x + z)| \, |z|  \int_{s=0}^1  \ind_{\{ \frac{|z|}{s} \le |\hat x_0-x|+R\} }  \frac{\dd s}{s^{d+1}}\, \dd z\\
      &\le R^{-d} \|\chi\|_{\infty}
        \int_{x + z \in B} \ind_{\{|z| \le 2R\} }  |\nabla u (x + z)| \, |z|  \, \int_{s= \frac{|z|}{3R}}^1   \frac{\dd s}{s^{d+1}}\, \dd z\\
      &\le  \frac{3^d}d\, \|\chi\|_{\infty} \, \int_{x + z \in B}  |\nabla u (x + z)| \, |z|^{1-d}  \,  \ind_{\{|z| \le 2R\} }\, \dd z ,
    \end{aligned}
  \end{equation*}
  so that using Young's inequality for convolutions, we end up, when
  $\frac1q + 1 = \frac1p + \frac1r$, with
  \begin{equation} \label{presti2}
    \begin{aligned}
      \| u - \bar{u}\|_{\fsL^q(B)}
      &\le  \frac{3^d}d\, \|\nabla u\|_{\fsL^p(B)} \; \; \| z \mapsto |z|^{1-d} \ind_{|z| \le 2R}\|_{\fsL^r(\R^d)}  \\
      &\le C_{d,p,q}\,  R^{1-d \left( \frac{1}{p} - \frac{1}{q} \right)}\, \|\nabla u\|_{\fsL^p(B)} .
    \end{aligned}
  \end{equation}
  The statement of the Lemma is obtained by putting together estimates
  \eqref{presti} and \eqref{presti2}.
\end{proof}

 We write down a final Lemma before concluding the proof of Proposition \ref{inpppg}.
 
\begin{lemma}[Covering]\label{thm:interpolation:covering}
  Let $\Omega \subset \R^d$ be a bounded \(\fsC^2\) domain and \(R_0>0\) obtained by
  \cref{thm:star-shaped-radius}.  There exists a constant \(c_d>0\) depending
  only on the dimension $d$ such that for every \(A \in (0,\infty)\), $\alpha \in [0,1)$ and
  $u \in \fsW^{1,p}_{loc}(\Omega)$ for $p > \frac{d}{3 - \alpha}$, there exists a cover
  \(\mathcal B\) of balls $B$ of $\Omega$ satisfying the following properties:
  \begin{itemize}
    \item for every \(B = B_d(x_0,R) \in \mathcal B\), we have that \(R \le R_0\). 
\par 
Moreover,  if \(R=R_0\), it holds that
    \begin{equation*}
      \| \nabla u \|_{\fsL^p(B\cap\Omega)} \le
      A\, R^{-3+ \alpha + \frac{d}{p}} ,
    \end{equation*}
   and,  if \(R<R_0\), it holds that
    \begin{equation*}
      \| \nabla u \|_{\fsL^p(B\cap\Omega)} =
      A\, R^{-3+ \alpha + \frac{d}{p}}.
    \end{equation*}
    \item every point of \(\Omega\) is covered by at most \(c_d\) balls.
  \end{itemize}
\end{lemma}

\begin{proof}
  For a point \(x \in \Omega\), we consider the function
  \(R \mapsto R^{3 - \alpha -\frac{d}{p}} \| \nabla u \|_{\fsL^p(B_d(x,R)\cap\Omega)}\).  This function
  starts for \(R=0\) at zero, and is increasing.  Hence either
  \(\| \nabla u \|_{\fsL^p(B_d(x,R_0)\cap\Omega)} \le A\, R_0^{-3 + \alpha +\frac{d}{p}}\) and we set
  \(R(x)=R_0\), or there exists \(0<R(x)\le R_0\) such that
  \(\| \nabla u \|_{\fsL^p(B_d(x,R(x))\cap\Omega)} = A\, (R(x))^{-3 + \alpha +\frac{d}{p}}\).

  The result then follows from Besicovitch covering Theorem by considering the
  balls \(B_d(x,R(x))\) for each $x \in \Omega$.
\end{proof}

\begin{proof}[Proof of \cref{inpppg}]
  We consider (for $p,q, \alpha$ satisfying the conditions of the Proposition)
  functions \(u \in \fsW^{1,p}_{loc}(\Omega)\), $w \in \fsC^{\alpha}(\overline{\Omega} )$. Here and in the rest
  of the proof, in the case when $\alpha=0$, $\fsC^{\alpha}( \overline{\Omega} )$ is replaced by
  $\fsL^{\infty}(\Omega)$ (and the corresponding norm becomes the $\|\,\,\|_{\infty}$ norm).

  Then we take \(A := \| w \|_{\fsC^{\alpha}( \overline{\Omega} )}\) and use
  \cref{thm:interpolation:covering} in order to get a covering \(\mathcal B\)
  satisfying the conclusion of \cref{thm:interpolation:covering}.  By the
  covering, we find that
  \begin{equation*}
    \| u \|_{\fsL^q(\Omega)}^q \le  \sum_{B \in \mathcal B} \| u \|_{\fsL^q(B\cap\Omega)}^q
    \le C_{d,p,q, \alpha}^q
    \sum_{B \in \mathcal B}
    \left[
      R^{1-d \left( \frac{1}{p}-\frac{1}{q} \right)} \,  \| \nabla u \|_{\fsL^p(B\cap\Omega)}
      + R^{- 2 + \alpha + \frac{d}q} \,  \| w \|_{\fsC^{\alpha}(\overline{B\cap\Omega})}
    \right]^q,
  \end{equation*}
  where we applied in the last step \eqref{eq:interpolation:ball-estimate-alpha} in
  \cref{thm:estimate-ball} to each cut ball $B \cap \Omega$, with \(B \in \mathcal B\).

  For each ball \(B \in \mathcal B\), the choice of \(R\) implies in the case that
  \(R < R_0\) that
  \begin{equation*}
    R^{1-d \left( \frac{1}{p}-\frac{1}{q} \right)} \,  \| \nabla u \|_{\fsL^p(B\cap\Omega)}
    = R^{- 2 + \alpha + \frac{d}q} \,  \| w \|_{\fsC^{\alpha}(  \overline{B\cap\Omega} )}
    = \| w \|_{\fsC^{\alpha}( \overline{\Omega} )}^{\frac{1-d \left( \frac{1}{p}-\frac{1}{q} \right)}{3- \alpha - d/p}}
    \|\nabla u \|_{\fsL^p(B\cap\Omega)}^{\frac{2 - \alpha-d/q}{3 - \alpha -d/p}}.
  \end{equation*}
  The case \(R = R_0\) can only apply a fixed number of times as each such ball
  covers at least a fixed volume and by the covering the volume of all balls is
  bounded (as the domain is bounded).  In this case
  \begin{equation*}
    R^{1-d \left( \frac{1}{p}-\frac{1}{q} \right)} \,  \| \nabla u \|_{\fsL^p(B\cap\Omega)}
    + R^{- 2 + \alpha + \frac{d}q} \,  \| w \|_{\fsC^{\alpha}( \overline{B\cap\Omega} )}
    \le
    \| w \|_{\fsC^{\alpha}( \overline{\Omega} )}^{\frac{1-d \left( \frac{1}{p}-\frac{1}{q} \right)}{3- \alpha - d/p}}
    \|\nabla u \|_{\fsL^p(B\cap\Omega)}^{\frac{2 - \alpha-d/q}{3 - \alpha -d/p}}
    + R_0^{-2+\alpha+\frac{d}{q}} \| w \|_{\fsC^{\alpha}( \overline{B\cap\Omega} )}.
  \end{equation*}
  Hence we find by the covering that (for a constant $C>0$)
  \begin{equation*}
    \| u \|_{\fsL^q(\Omega)}^q
    \le C\, \| w \|_{\fsC^{\alpha}( \overline{\Omega} )}^{\frac{q-dq \left( \frac{1}{p}-\frac{1}{q} \right)}{3 - \alpha -d/p}}
    \sum_{B \in \mathcal B}
    \|\nabla u \|_{\fsL^p(B\cap\Omega)}^{q \left(\frac{2- \alpha - d/q}{3- \alpha - d/p}\right)}
    +
  C\,  \| w \|_{\fsC^{\alpha}( \overline{\Omega} )}^q.
  \end{equation*}
  As the covering covers every point at most \(c_d\) times, and observing that
  (thanks to the assumption on $p,q, \alpha$ in the Proposition)
  \begin{equation}\label{eq:interpolation:condition-sum}
    q \left( \frac{2 -  \alpha - \frac{d}{q}}{3- \alpha - \frac{d}{p}} \right) \ge p,
  \end{equation}
  we see that
  \begin{equation*}
    \sum_{B \in \mathcal B}
    \|\nabla u \|_{\fsL^p(B\cap\Omega)}^{q \left(\frac{2- \alpha - d/q}{3- \alpha - d/p}\right)}
    \le  \sum_{B \in \mathcal B}
    \|\nabla u \|_{\fsL^p(B\cap\Omega)}^{p} \,\,   \|\nabla u \|_{\fsL^p(\Omega)}^{q \left(\frac{2- \alpha - d/q}{3- \alpha - d/p}\right) -p}
    \le c_d\,
    \|\nabla u \|_{\fsL^p(\Omega)}^{q \left(\frac{2- \alpha - d/q}{3- \alpha - d/p}\right)},
  \end{equation*}
  which yields the sought interpolation.
\end{proof}

\AtNextBibliography{\small}
\printbibliography
\end{document}